\documentclass[12pt]{article}
\usepackage{amsfonts}
\usepackage{amsmath, amsthm, full page, tikz-cd, comment,
amssymb, amscd, graphicx, enumitem, stmaryrd, enumitem
}
\usepackage{xcolor}
\allowdisplaybreaks
\newtheorem{thm}{Theorem}[section]
\newtheorem{defn}[thm]{Definition}
\newtheorem{prop}[thm]{Proposition}

\newtheorem{rema}[thm]{Remark}

\newcommand{\nn}{\nonumber \\}

 \newcommand{\res}{\mbox{\rm Res}}

\newcommand{\wt}{\mbox{\rm wt}\,}
\newcommand{\swt}{\mbox{\rm {\scriptsize wt}}\,}

\newcommand{\Y}{\mathcal{Y}}
\newcommand{\C}{\mathbb{C}}
\newcommand{\Z}{\mathbb{Z}}
\newcommand{\R}{\mathbb{R}}

\newcommand{\N}{\mathbb{N}}

\renewcommand{\i}{\mathbf{i}}

\renewcommand{\l}{\llfloor}
\renewcommand{\r}{\rrfloor}

 \makeatletter
\newlength{\@pxlwd} \newlength{\@rulewd} \newlength{\@pxlht}
\catcode`.=\active \catcode`B=\active \catcode`:=\active \catcode`|=\active
\def\sprite#1(#2,#3)[#4,#5]{
   \edef\@sprbox{\expandafter\@cdr\string#1\@nil @box}
   \expandafter\newsavebox\csname\@sprbox\endcsname
   \edef#1{\expandafter\usebox\csname\@sprbox\endcsname}
   \expandafter\setbox\csname\@sprbox\endcsname =\hbox\bgroup
   \vbox\bgroup
  \catcode`.=\active\catcode`B=\active\catcode`:=\active\catcode`|=\active
      \@pxlwd=#4 \divide\@pxlwd by #3 \@rulewd=\@pxlwd
      \@pxlht=#5 \divide\@pxlht by #2
      \def .{\hskip \@pxlwd \ignorespaces}
      \def B{\@ifnextchar B{\advance\@rulewd by \@pxlwd}{\vrule
         height \@pxlht width \@rulewd depth 0 pt \@rulewd=\@pxlwd}}
      \def :{\hbox\bgroup\vrule height \@pxlht width 0pt depth
0pt\ignorespaces}
      \def |{\vrule height \@pxlht width 0pt depth 0pt\egroup
         \prevdepth= -1000 pt}
   }
\def\endsprite{\egroup\egroup}
\catcode`.=12 \catcode`B=11 \catcode`:=12 \catcode`|=12\relax
\makeatother

\makeatletter
\newcommand{\raisemath}[1]{\mathpalette{\raisem@th{#1}}}
\newcommand{\raisem@th}[3]{\raisebox{#1}{$#2#3$}}
\makeatother

\makeatletter
\newcommand{\subalign}[1]{%
	\vcenter{%
		\Let@ \restore@math@cr \default@tag
		\baselineskip\fontdimen10 \scriptfont\tw@
		\advance\baselineskip\fontdimen12 \scriptfont\tw@
		\lineskip\thr@@\fontdimen8 \scriptfont\thr@@
		\lineskiplimit\lineskip
		\ialign{\hfil$\m@th\scriptstyle##$&$\m@th\scriptstyle{}##$\hfil\crcr
			#1\crcr
		}%
	}%
}
\makeatother

\def\hboxtr{\FormOfHboxtr} 
\sprite{\FormOfHboxtr}(25,25)[0.5 em, 1.2 ex] 

:BBBBBBBBBBBBBBBBBBBBBBBBB |
:BB......................B |
:B.B.....................B |
:B..B....................B |
:B...B...................B |
:B....B..................B |
:B.....B.................B |
:B......B................B |
:B.......B...............B |
:B........B..............B |
:B.........B.............B |
:B..........B............B |
:B...........B...........B |
:B............B..........B |
:B.............B.........B |
:B..............B........B |
:B...............B.......B |
:B................B......B |
:B.................B.....B |
:B..................B....B |
:B...................B...B |
:B....................B..B |
:B.....................B.B |
:B......................BB |
:BBBBBBBBBBBBBBBBBBBBBBBBB |

\endsprite

\sprite{\FormOfShboxtr}(25,25)[0.3 em, 0.72 ex] 

:BBBBBBBBBBBBBBBBBBBBBBBBB |
:BB......................B |
:B.B.....................B |
:B..B....................B |
:B...B...................B |
:B....B..................B |
:B.....B.................B |
:B......B................B |
:B.......B...............B |
:B........B..............B |
:B.........B.............B |
:B..........B............B |
:B...........B...........B |
:B............B..........B |
:B.............B.........B |
:B..............B........B |
:B...............B.......B |
:B................B......B |
:B.................B.....B |
:B..................B....B |
:B...................B...B |
:B....................B..B |
:B.....................B.B |
:B......................BB |
:BBBBBBBBBBBBBBBBBBBBBBBBB |

\endsprite

\title{ {\bf  Cofiniteness and $P(z)$-tensor product bifunctors  
in orbifold theories associated to  abelian but not-necessarily-finite groups} }
\date{}
\author{Yi-Zhi Huang}

\begin{document}

\bibliographystyle{alpha}
\maketitle
\begin{abstract}
Let $V$ be a M\"{o}bius vertex algebra and $G$ an abelian group of automorphisms of 
$V$. We construct $P(z)$-tensor product bifunctors for 
the category of $C_{n}$-cofinite grading-restricted generalized 
$g$-twisted $V$-modules (without $g$-actions) for $g\in G$
and the category of  $C_{n}$-cofinite grading-restricted generalized 
$g$-twisted $V$-modules with $G$-actions for $g\in G$. In this paper, an automorphism $g$ of 
$V$ can be of infinite order and does not have to act  semisimply on $V$, and the group $G$ can be an infinite
abelian group containing nonsemisimple automorphisms of $V$. 
\end{abstract}

\renewcommand{\theequation}{\thesection.\arabic{equation}}
\renewcommand{\thethm}{\thesection.\arabic{thm}}
\setcounter{thm}{0}
\setcounter{equation}{0}
\section{Introduction}

Two-dimensional orbifold conformal field theories play an important role in the construction,
classification and applications of two-dimensional conformal field theories. 
The moonshine module $V^{\natural}$ constructed in mathematics by 
Frenkel, Lepowsky and Meurman \cite{FLM1} \cite{FLM2} \cite{FLM3}
is the first example of 
two-dimensional orbifold conformal field theories. The systematic study of two-dimensional 
orbifold conformal field theories in physics was started by Dixon, 
Harvey, Vafa, and Witten  
\cite{DHVW1} \cite{DHVW2}. In \cite{H-log-twisted-mod}, the author introduced 
$g$-twisted modules for a vertex operator algebra $V$ 
for an automorphism of $V$ of infinite order. 
In \cite{H-orbifold}, the author initiated a program to construct 
general two-dimensional orbifold conformal field theories starting from a vertex operator
algebra $V$ and a group $G$ of automorphisms of $V$. Here the group $G$ is not necessarily 
finite abelian. 
A number of conjectures and open problems,
including those on the convergence, associativity, commutativity, modular invariance 
of twisted intertwining operators, and on the corresponding $G$-crossed 
braided tensor category structures were formulated and discussed in \cite{H-orbifold}. 

Recently,  progress in this program have been made  by Du and the author \cite{DH} on 
twisted intertwining operators and $P(z)$-tensor product bifunctors for categories of suitable 
twisted modules for a vertex operator algebra $V$ under certain
assumptions (see below for more discussions on these assumptions,
by Tan on the cofiniteness of suitable twisted $V$-modules \cite{T2} and on differential equations 
satisfied by products of twisted intertwining operators in the mostly untwisted case \cite{T1} 
and in the finite abelian case \cite{T3}, and 
by Du on the associativity of twisted intertwining operators under suitable 
convergence and extension assumptions \cite{D}. The present paper is also one step in this program. 

We want to emphasize the importance of $P(z)$-tensor product bifunctors. 
In the untwisted case, the braided and even modular tensor 
category structures has been constructed on suitable  categories 
of modules for a vertex operator algebra satisfying certain conditions
(see \cite{tensor1}--\cite{tensor3}, \cite{tensor4}, \cite{H-rigidity},
\cite{H-finite-length}, 
\cite{HLZ1}--\cite{HLZ8}, \cite{H-C1-vtc}). However, from these braided or even modular tensor 
category structures, we cannot reconstruct the vertex operator algebra.
For example, the modular tensor category for the moonshine module 
$V^{\natural}$ is equivalent to the tensor category of finite-dimensional 
vector spaces over $\C$. It is impossible to reconstruct the moonshine module from 
this trivial tensor category. On the other hand, the constructions in 
\cite{tensor1}--\cite{tensor3}, \cite{tensor4}, \cite{H-rigidity}, \cite{H-finite-length}, 
\cite{HLZ1}--\cite{HLZ8}, and \cite{H-C1-vtc} give vertex tensor categories (see 
\cite{HL-vtc} for a definition of vertex tensor category). It is a conjecture 
that one can reconstruct vertex operator algebras from vertex tensor categories. 
In the twisted case,  it is conjectured in \cite{H-orbifold} that the category of grading-restricted generalized 
$g$-twisted modules for a $C_{2}$-cofinite vertex operator algebra of positive energy and 
for $g$ in a finite group $G$ of automorphisms has a natural $G$-crossed braided tensor category
structure  in the sense of 
of \cite{Tu}. To prove this conjecture, we have to construct  first a $G$-crossed vertex tensor category
structure and then use the same procedure as in the untwisted case to obtain the 
$G$-crossed braided tensor category structure. 
$P(z)$-tensor product bifunctors
are one set of the most important data of a $G$-crossed vertex tensor category.

In \cite{DH},  $P(z)$-tensor product bifunctors for suitable categories of grading-restricted 
generalized $V$-modules are constructed based on suitable assumptions
(see Assumption 4.4 in \cite{DH}).  The results in \cite{DH} are twisted generalizations of 
the results in the untwisted case in \cite{tensor1}, \cite{tensor3}, \cite{HLZ3} and \cite{HLZ4}.
Theorem 4.8 and Corollary 4.9 in \cite{DH} give some strong conditions 
on categories of twisted $V$-modules such that Assumption 4.4 in \cite{DH} holds for such categories. 
But these conditions are very strong and are not easy to verify. 
We would like to find conditions on a vertex operator algebra $V$ and categories of suitable twisted modules 
that are easy to verify or have been verified in the literature such that $P(z)$-tensor product bifunctors
exist in the category. 

In this paper, we construct $P(z)$-tensor product bifunctors on 
categories of  suitable twisted modules satisfying some cofiniteness conditions
in the case that the automorphisms
involved commute but might be of infinite orders. 
Instead of trying to generalize the conditions in Theorem 4.8 and Corollary 4.9 in \cite{DH},
in this paper, we prove generalizations and analogues of some results 
in \cite{H-C1-vtc} for lower-bounded generalized twisted $V$-modules 
and twisted intertwining operators among such $g$-twisted $V$-modules for $g$
 in an abelian group $G$ of automorphisms 
of $V$.  In \cite{H-C1-vtc}, a precisely formulated generalization of an inequality of Nahm \cite{N} is proved and 
$P(z)$-tensor rpoduct bifunctors are constructed 
using some results in  \cite{HLZ3} and \cite{HLZ4} and this inequality. In fact, in \cite{N}, Nahm introduced 
a notion of quasi-rational module for a $\mathcal{W}$-algebra.  In mathematics, 
$\mathcal{W}$-algebras are suitable vertex operator algebras and 
quasi-rational modules means $C_{1}$-cofinite modules. 
For a suitable module $W$ for a vertex operator aLgebra $V$ with the vertex operator map $Y_{W}$,
the $C_{1}$-subspace  $C_{1}(W)$ of $W$ is the subspace of $V$ spanned by 
$\res_{x}x^{-1}Y_{W}(v, x)w$ for $v\in V_{+}=\coprod_{n\in \Z_{+}}V_{(n)}$ and $w\in W$.
The $V$-module $W$ is $C_{1}$-cofinite if $\dim W/C_{1}(W)<\infty$.
Assuming that  a suitable fusion product $W_{12}$ 
of two $C_{1}$-cofinite $V$-modules $W_{1}$ and $W_{2}$ 
exists, Nahm derived in \cite{N}  an inequality
\begin{equation}\label{Nahm-ineq}
\dim (W_{12}/C_{1}(W_{12}))
\le \dim (W_{1}/C_{1}(W_{1})) \dim (W_{2}/C_{1}(W_{2})).
\end{equation}
Nahm did not give a construction of a fusion product in \cite{N}. 
Instead, he derived the inequality \eqref{Nahm-ineq} from some basic 
physical assumptions on conformal field theories. A generalization of 
the inequality \eqref{Nahm-ineq} to surjective intertwining operators is formulated precisely and 
proved mathematically  in \cite{H-C1-vtc}. In particular, if 
$W_{12}$ is taken to be the $P(z)$-tensor product $W_{1}\boxtimes_{P(z)}W_{2}$
constructed in  \cite{H-C1-vtc}, the inequality \eqref{Nahm-ineq} holds. 
In \cite{YZ}, Yang and Zhu proved a twisted Nahm inequality for 
quotients of lower-bounded generalized twisted $V$-modules 
by their $C_{2}$-subspaces\footnote{In \cite{YZ},  for a lower-bounded generalized 
$g$-twisted $V$-module for finite-order $g$, a subspace $C_{1}(W)$ of $W$ is introduced. In fact, $C_{1}(W)$ in \cite{YZ}
is $C_{2}(W)$ in \cite{T1} and  the present paper, and $C_{1}$-cofiniteness in \cite{YZ} is 
called $C_{2}$-cofiniteness in \cite{T1} and the present paper. There is also a notion of $C_{1}$-cofiniteness in 
\cite{T1}. Note that 
by the main result proved by Tan in \cite{T1},  finitely generated lower-bounded generalized twisted $V$-modules 
are $C_{2}$-cofinite in the sense of \cite{T1} and the present paper if $V$ is $C_{2}$-cofinite and of positive energy. 
Because of this result
and the examples of twisted modules for the Heisenberg vertex operator algebras,
in the author's opinion,  $C_{2}(W)$ and $C_{2}$-cofinite are the correct notation and term to use.}
 (see below for the definition)
in the case that the automorphisms involved commute and are of finite order. 
Using this inequality, it is proved in \cite{YZ} that the category of $C_{2}$-cofinite 
lower-bounded generalized twisted $V$-modules twisted by commuting automorphisms  of finite orders 
is closed under a suitable fusion product. 

In this paper, we construct $P(z)$-tensor product bifunctors for categories of 
$C_{n}$-cofinite grading-restricted generalized $g$-twisted $V$-modules for 
$g$ in an abelain group $G$ of automorphisms of $V$. The group $G$ does not have
to be finite and can contain automorphisms that do not act semisimply on $V$. 
Our method is the same as the one in \cite{H-C1-vtc}. 

Here we give more detailed descriptions of our results. 
Let $V$ be a grading-restricted vertex algebra. For an automorphism $g$ of $V$,
we have $V=\coprod_{\alpha\in P_{V}^{g}}V^{[\alpha]}$, where $P_{V}^{g}\subset [0, 1)+\i\R\subset \C$
and $V^{[\alpha]}$ for $\alpha\in P_{V}^{g}$ is the generalized eigenspace of $g$ with 
the eigenvalue $e^{2\pi \i \alpha}$. For $n\in 2+\N$ and a generalized $g$-twisted $V$-module 
$W$ with the twisted vertex operator map $Y_{W}$, let the $C_{n}$-subspace $C_{n}(W)$ 
of $W$ be the subspace of 
$W$ spanned by elements of the form 
$\res_{x}x^{\alpha-n}Y_{W}(v, x)w$ for $v\in V^{[\alpha]}$, $\alpha\in P_{V}^{g}$, and $w\in W$.  Let
$g_{1}$ and $g_{2}$ be 
commuting automorphisms of $V$,  $W_{1}$ and $W_{2}$ lower-bounded generalized 
$g_{1}$- and $g_{2}$-twisted $V$-modules, respectively, and
$W_{3}$ a generalized $g_{1}g_{2}$-twisted $V$-module. Assuming that 
there exists a  surjective twisted intertwining operator of type $\binom{W_{3}}{W_{1}W_{2}}$,
we prove 
\begin{equation}\label{twisted-Cn-ineq}
\dim (W_{3}/C_{\min(p, q)}(W_{3}))\le \dim (W_{1}/C_{p}(W_{1}))
\dim (W_{2}/C_{q}(W_{2}))
\end{equation}
 for $p, q\in 2+\N$.  Here $g_{1}$ and $g_{2}$ can be of infinite order and do not have to act
semisimply on $V$. 
This inequality is proved using the same method as in \cite{H-C1-vtc}, but
with the Jacobi identity for (untwisted) intertwining operators  in \cite{H-C1-vtc} replaced by 
a Jacobi identity for twisted intertwining operators
in the case that the automorphisms of $V$ involved commute with each other. 

For a generalized twisted $V$-module, we say that
$W$ is $C_{n}$-cofinite if $\dim W/C_{n}(W)<\infty$. We now assume that $V$ has in addition  a M\"{o}bius vertex 
algebra or quasi-vertex operator algebra structure (see \cite{FHL} and \cite{HLZ1}). In particular, we have 
an operator $L_{V}(1)$ on $V$ and for a lower-bounded generalized twisted $V$-mdoule $W$, 
we have a contragredient $W'$ (see 
 \cite{H-twisted-int}). 
Let $G$ be an abelian group of automorphisms of $V$. 
Using some results in \cite{DH},  the inequality above  in the case $p=q=n\in 2+\N$, and other results proved in this paper
on $C_{n}$-cofinite lower-bounded generalized twisted $V$-modules and twisted intertwining operators, 
we construct $P(z)$-tensor product bifunctors for the category of $C_{n}$-cofinite grading-restricted generalized 
$g$-twisted $V$-modules (without $g$-actions) for $g\in G$
and the category of  $C_{n}$-cofinite grading-restricted generalized 
$g$-twisted $V$-modules with $G$-actions for $g\in G$. 
Though the second category is in fact a subcategory of the first category, 
the $P(z)$-tensor product bifunctor for the second category is not the restriction of 
the $P(z)$-tensor product bifunctor for the first category to the second tensor category. 
We also emphasize again that 
in this paper, an automorphism $g$ of 
$V$ can be of infinite order and does not have to act semisimply on $V$, and the group $G$ can be an infinite
abelian group containing nonsemisimple automorphisms of $V$.

However, even in the abelian case discussed in this paper, 
the results and the method in \cite{H-C1-vtc} cannot be generalized directly to obtain results on 
$C_{1}$-cofinite lower-bounded generalized twisted $V$-modules (see, for example, 
\cite{T1} for a definition of $C_{1}$-cofinite twisted modules) and the corresponding 
twisted intertwining operators. So 
the $C_{1}$-cofinite case is still an open problem.

Since Tan proved that for a $C_{2}$-cofinite vertex operator algebra $V$ of  positive energy,
finitely-generated weak twisted  $V$-modules are $C_{n}$-cofinite for $n\ge 2$, the construction
of the $P(z)$-tensor product bifunctors in this paper can be applied to module categories 
for many vertex operator algebras 
that have been proved to be $C_{2}$-cofinite. One important class of examples 
is  the simple affine  vertex operator algebras 
at positive integral levels. They are $C_{2}$-cofinite and the Lie groups corresponding to the finite-dimensional 
simple Lie algebras are groups of automorphisms of these affine vertex operator algebras. 
The results of the present paper  can be applied to such a vertex operator algebra 
and an infinite abelian subgroup of the corresponding Lie group. Another class of examples is 
related to $\mathcal{W}$-algebras obtained as the kernels of suitable screening operators. 
In fact, exponentials of screening operators are automorphisms of infinite orders of suitable vertex operator algebras 
and the kernels of these screening operators are exactly the fixed point subalgebras 
of the vertex operator algebras that one starts with (see \cite{H-log-twisted-mod}). The results 
of the present paper can be used to initiate a study of  
such $\mathcal{W}$-algebras using the orbifold theory associated to infinite abelian groups. 

The present paper is organized as follows: In the next section (Section 2), we give the 
definitions of various notions of twisted modules and twisted intertwining operators
among twisted modules twisted by commuting automorphisms. 
We define twisted intertwining operators using a Jacobi identity, which is equivalent to the 
duality property in \cite{DH} in this abelian case (but we will prove this equivalence in another paper). 
In  Section 3, we prove some basic properties of 
 lower-bounded generalized twisted modules. The inequality \eqref{twisted-Cn-ineq} is proved in
Section 4. The constructions of the $P(z)$-tensor product bifunctors 
are given in Section 5: In Subsection 5.1, we give the construction of 
the $P(z)$-tensor product bifunctors for the category of $C_{n}$-cofinite grading-restricted generalized 
$g$-twisted $V$-modules (without $g$-actions) for $g\in G$.
In Subsection 5.2, we give the construction of 
the $P(z)$-tensor product bifunctors for the category of $C_{n}$-cofinite grading-restricted generalized 
$g$-twisted $V$-modules with $G$-actions for $g\in G$.

In this paper, we fix a grading-restricted vertex algebra $V$.  In Section 5, we will assume in addition
that $V$ is a M\"{o}bius vertex algebra or quasi-vertex operator algebra (see \cite{FHL} and \cite{HLZ1}).
Since we sometimes use $i$ as an index, we will denote the imaginary number $\sqrt{-1}$ by $\i$. 

\paragraph{Acknowledgment} 
The author is grateful to Jishen Du and Daniel Tan for many discussions on 
twisted modules, twisted intertwining operators, and orbifold theories.

\setcounter{equation}{0}
\section{Twisted modules and twisted intertwining operators}

We give the definitions of various notions of twisted module, module maps between these twisted modules,
and twisted intertwining operators among various types of twisted modules twisted by commuting 
automorphisms of $V$ in this section. 
We define twisted intertwining operators in this case using a Jacobi identity. We then derive 
a suitable component form of the Jacobi identity that will be needed in later sections. 

For an automorphism $g$ of $V$, $V$ can be decomposed as a direct sum of generalized eigenspaces
of  $g$. An eigenvalue of $g$ on $V$ is of the form $e^{2\pi i \alpha}$
for some $\alpha\in \C$. Since $e^{2\pi i (\alpha+m)}=e^{2\pi i \alpha}$ for $m\in \Z$, we can always find 
such an $\alpha\in \C$ satisfying $\Re(\alpha)\in [0, 1)$. 

For $\alpha, \beta\in \C$ such that $\Re(\alpha), \Re(\beta)\in [0, 1)$, either $\Re(\alpha+\beta)\in [0, 1)$ or 
$\Re(\alpha+\beta-1)\in [0, 1)$. Let
$$\epsilon(\alpha, \beta)=\left\{\begin{array}{ll}0&\Re(\alpha+\beta)\in [0, 1)\\
1&\Re(\alpha+\beta-1)\in [0, 1).\end{array}\right.$$
Then for $\alpha, \beta\in \C$ such that $\Re(\alpha), \Re(\beta)\in [0, 1)$, 
we have $\Re(\alpha+\beta-\epsilon(\alpha, \beta))\in [0, 1)$.
In general, for $\alpha_{1}, \dots, \alpha_{k}\in \C$ such that $\Re(\alpha_{1}), \dots, \Re(\alpha_{k})\in [0, 1)$, 
there exists a unique 
$\epsilon(\alpha_{1}, \dots, \alpha_{k})
\in \N$ satisfying $\epsilon(\alpha_{1}, \dots, \alpha_{k})\le k-1$
 such that $\Re(\alpha_{1}+\dots +\alpha_{k}-\epsilon(\alpha_{1}, \dots, \alpha_{k}))\in [0, 1)$.
For simplicity, we also denote $\alpha_{1}+\dots +\alpha_{k}-\epsilon(\alpha_{1}, \dots, \alpha_{k})$
by $\sigma(\alpha_{1}, \dots, \alpha_{k})$. 

Let 
$$P_{V}^{g}=\{\alpha\in \C\mid 
\Re(\alpha)\in [0, 1), e^{2\pi \i \alpha}\;\text{is an eigenvalues of}\; g\}.$$
For $\alpha\in P_{V}^{g}$, we use 
$V^{[\alpha]}$ to denote the generalized eigenspace of $g$ with eigenvalue $e^{2\pi \i \alpha}$.
Then 
$$V=\coprod_{\alpha\in P_{V}^{g}}V^{[\alpha]}=\coprod_{n\in \Z}\coprod_{\alpha\in P_{V}^{g}}
V_{(n)}^{[\alpha]},$$
where $V_{(n)}^{[\alpha]}=V_{(n)}\cap V^{[\alpha]}$ for $n\in \Z$ and $\alpha\in P_{V}^{g}$. 

From Section 2 of \cite{HY} and Lemma 3.4 in \cite{H-twist-vo}, we know that there exist a semisimple operator
$\mathcal{S}_{g}$ and a local nilpotent operator $\mathcal{N}_{g}$ on $V$ such that 
$g=e^{2\pi \i(\mathcal{S}_{g}+\mathcal{N}_{g})}=e^{2\pi \i\mathcal{L}_{g}}$, 
where as in \cite{H-twist-vo}, $\mathcal{L}_{g}=\mathcal{S}_{g}+\mathcal{N}_{g}$,
 and $V^{[\alpha]}$ is the eigenspace 
of $\mathcal{S}_{g}$ with eigenvalue $\alpha$ for $\alpha\in P_{V}^{g}$. 
For $\alpha, \beta\in P_{V}^{g}$,
 $u\in V^{[\alpha]}$, and $v\in V^{[\beta]}$, since $e^{2\pi \i \mathcal{S}_{g}}$ is also an automorphism of 
$V$ by Proposition 3.5 in \cite{H-twist-vo},
 we have $e^{2\pi \i \mathcal{S}_{g}}Y_{V}(u,  x)v=Y_{V}(e^{2\pi \i \mathcal{S}_{g}}u, x)e^{2\pi \i \mathcal{S}_{g}}v
=e^{2\pi \i (\alpha+\beta)}Y_{V}(u,  x)v$. If there exist $u\in V^{[\alpha]}$, $v\in V^{[\beta]}$ such that 
$Y_{V}(u, x)v\ne 0$, then there exists $n\in \Z$ such that $u_{n}v=\res_{x}x^{n}Y_{V}(u, x)v\ne 0$ 
is a generalized eigenvector of 
$g$ with eigenvalue $e^{2\pi \i (\alpha+\beta)}$. In this case, if $\Re(\alpha+\beta)\in [0, 1)$, then
$\alpha+\beta\in P_{V}^{g}$ and if $\Re(\alpha+\beta-1)\in [0, 1)$, $\alpha+\beta-1\in P_{V}^{g}$.

Let $g_{1}$ and $g_{2}$ be automorphisms of $V$. Then we have 
$g_{1}=e^{2\pi \i(\mathcal{S}_{g_{1}}+\mathcal{N}_{g_{1}})}=e^{2\pi \i\mathcal{L}_{g_{1}}}$ and 
$g_{1}=e^{2\pi \i(\mathcal{S}_{g_{2}}+\mathcal{N}_{g_{2}})}=e^{2\pi \i\mathcal{L}_{g_{2}}}$. 
Since $g_{1}$ and $g_{2}$ commute, 
the operators $\mathcal{S}_{g_{1}}$, $\mathcal{N}_{g_{1}}$, $\mathcal{S}_{g_{2}}$,
and $\mathcal{N}_{g_{2}}$ also commute. In fact, since $g_{1}$ and $g_{2}$ commute,  $V$ 
can be decomposed as a direct sum of common generalized eigenspaces $V^{[\alpha_{1}, \alpha_{2}]}$
for $\alpha_{1}\in P_{V}^{g_{1}},  \alpha_{2}\in P_{V}^{g_{2}}$
of $g_{1}$ and $g_{2}$ with eigenvalues $e^{2\pi \i \alpha_{1}}$ and $e^{2\pi \i \alpha_{2}}$,
respectively. Then $\mathcal{S}_{g_{1}}$ and $\mathcal{S}_{g_{2}}$ act on $V^{[\alpha_{1}, \alpha_{2}]}$
as the numbers $\alpha_{1}$ and $\alpha_{2}$, respectively. In particular, 
$\mathcal{S}_{g_{1}}$ and $\mathcal{S}_{g_{2}}$ commute. On the other hand, since 
$V^{[\alpha_{1}, \alpha_{2}]}$ is invariant under $g_{1}$ and $g_{2}$, 
$\mathcal{S}_{g_{1}}$ and $\mathcal{S}_{g_{2}}$ act on $g_{1}v$ or $g_{2}v$ for $v\in V^{[\alpha_{1}, \alpha_{2}]}$
are equal to $g_{1}v$ or $g_{2}v$ multiplied by $\alpha_{1}$ and $\alpha_{2}$, respectively. 
Thus we see that $\mathcal{S}_{g_{1}}$ and $\mathcal{S}_{g_{2}}$ , $g_{1}$, and $g_{2}$  commute.
and in particular, 
 $\mathcal{S}_{g_{1}}$, $\mathcal{S}_{g_{2}}$ , $\mathcal{L}_{g_{1}}$, and $\mathcal{L}_{g_{2}}$
commute.
Then  $\mathcal{S}_{g_{1}}$, $\mathcal{S}_{g_{2}}$, $\mathcal{N}_{g_{1}}=\mathcal{L}_{g_{1}}-\mathcal{S}_{g_{1}}$,
and $\mathcal{N}_{g_{2}}=\mathcal{L}_{g_{2}}-\mathcal{S}_{g_{2}}$ commute.

We first give the definitions of various $g$-twisted $V$-modules for an automorphism
$g$ of $V$. The definitions below 
are generalizations or modifications of the definitions of suitable  $g$-twisted $V$-modules 
given in \cite{H-log-twisted-mod}, \cite{HY}, and \cite{H-twist-vo}. It is important to note that 
$g$ can be of infinite order and does not have to act  semisimply on $V$. 

\begin{defn}\label{defn-twisted-mod}
{\rm  A {\it weak  $g$-twisted 
$V$-module without a $g$-action}  is a vector space $W$
 equipped with operators $L_{W}(0)$ and $L_{W}(-1)$ on $W$,  and a linear map
\begin{eqnarray*}
Y_{W}^g: V\otimes W &\to& W\{x\}[\log  x],\\
v \otimes w &\mapsto &Y_{W}^g(v, x)w=\sum_{n\in \C}\sum_{k\in \N}
v_{n, k}x^{-n-1}(\log x)^{k}
\end{eqnarray*}
called {\it twisted vertex operator map}, satisfying the following conditions:
\begin{enumerate}

\item The {\it lower-truncation condition}: For $v\in V$ and $w\in W$, $v_{n, k}w=0$ for $\Re(n)$ sufficiently large. 

\item The {\it equivariance property}: For $p \in \mathbb{Z}$, $z
\in \mathbb{C}^{\times}$,  $v \in V$ and $w \in W$, 
$$(Y^{g}_{W})^{p + 1}(gv,
z)w = (Y^{g}_{W})^{p}(v, z)w,$$
where for $p \in \mathbb{Z}$, $(Y^{g}_{W})^{ p}(v, z)$
is the $p$-th analytic branch of $Y_{W}^g(v, x)$.

\item The {\it identity property}: For $w \in W$, $Y_{W}^g({\bf 1}, x)w
= w$.

\item The {\it Jacobi identity}: For $u, v\in V$, 
\begin{align}\label{jacobi-1}
x_0^{-1}&\delta\left(\frac{x_1 - x_2}{x_0}\right)
Y_{W}^{g}(u, x_1)
Y_{W}^{g}(v, x_2)-  x_0^{-1}\delta\left(\frac{- x_2 + x_1}{x_0}\right)
Y_{W}^{g}(v, x_2)
Y_{W}^{g}(u, x_1)\nn
&= x_1^{-1}\delta\left(\frac{x_2+x_0}{x_1}\right)
Y_{W}^{g}\left(Y_{V}\left(\left(\frac{x_2+x_0}{x_1}\right)^{\mathcal{L}_{g}}
u, x_0\right)v, x_2\right).
\end{align}

\item The {\it $L(0)$-commutator formula}: 
$$[L_{W}(0), Y_{W}(v, x)]=z\frac{d}{dx}Y_{W}(v, x)+Y_{W}(L_{V}(0)v, x)$$
for $v\in V$. 

\item The  {\it $L(-1)$-derivative property} and {\it $L(-1)$-commutator formula}: For $v \in V$,
\[
\frac{d}{dx}Y_{W}(v, x) =Y_{W}(L_{V}(-1)v, x)=[L_{W}(-1), Y_{W}(v, x)].
\]

\end{enumerate}
A {\it weak $g$-twisted $V$-module with a $g$-action} is a weak 
$g$-twisted $V$-module without a $g$-action equipped with a $\langle g\rangle$-action (where $\langle g\rangle$ 
denote the cyclic group generated by $g$) satisfying the following 
{\it $g$-compatibility condition}: The $g$-action on $W$ commutes with
$L_{W}(0)$ and $L_{W}(-1)$ and
$gY_{W}(u,x)w=Y_{W}(gu,x)gw$  for $u\in V$ and $w\in W$.
Let $G$ be a group of automorphisms of $V$ such that $g\in G$. A {\it weak $g$-twisted $V$-module with 
a $G$-action} is a weak $g$-twisted $V$-module $W$ with a $g$-action equipped with a
$G$-module structure extending the $g$-action satisfying 
 the  following {\it $G$-compatibility condition}:  The $G$-action on $W$ commutes with $L_{W}(0)$,
$L_{W}(-1)$ and $hY_{W}(u,x)w=Y_{W}(hu,x)hw$  for $h\in G$, $u\in V$ and $w\in W$.
A {\it generalized  $g$-twisted 
$V$-module without a $g$-action} is a weak $g$-twisted $V$-module  without a $g$-action
with a ${\C}$-grading
$$W = \coprod_{n \in \C}W_{[n]}$$
 (graded by 
{\it weights}) satisfying the $g$-compatibility condition above and 
the following {\it $L(0)$-grading condition}:
For $w\in W_{[n]}=\coprod_{\alpha\in P_{W}^{g}} W_{[n]}^{[\alpha]}$, $n\in \C$, there exists
$K\in \Z_{+}$ such that 
$(L_{W}(0)-n)^{K}w=0$. 
A {\it generalized  $g$-twisted 
$V$-module with a $g$-action} 
is a generalized  $g$-twisted 
$V$-module without a $g$-action with a $\C\times \C/\Z$-grading
$$W = \coprod_{n \in \C}\coprod_{\alpha\in \C/\Z} W_{[n]}^{[\alpha]}$$
 (graded by 
{\it weights} and {\it $g$-weights}) satisfying the following 
 {\it $g$-grading condition}: For $\alpha\in \C/\Z$,
$w \in W^{[\alpha]}=\coprod_{n\in \C}W_{[n]}^{[\alpha]}$, there exist $\Lambda\in \Z_{+}$ 
such that $(g-e^{2\pi \i \alpha})^{\Lambda}w=0$. 
A {\it lower-bounded  generalized $g$-twisted $V$-module with or without a $g$-action} 
is a generalized $g$-twisted
$V$-module  $W$ with or without a $g$-action, respectively, such that
$W_{[n]} = 0$ when $\Re(n)<N$  for some $N\in \Z$. A 
{\it grading-restricted generalized $g$-twisted $V$-module with or without a $g$-action}  is 
a lower-bounded generalized $g$-twisted $V$-module $W$ with or without a $g$-action, respectively,  such that  for each $n \in
\mathbb{C}$, $\dim W_{[n]}<\infty$.  A {\it quasi-finite-dimensional  generalized $g$-twisted $V$-module
with or without a $g$-action}
is a generalized $V$-module $W$ 
with or without a $g$-action, respectively, such that for $N\in \Z$,  $\dim\left(\coprod_{\Re(n) \le N}
\coprod_{\alpha\in \C/\Z} W_{[n]}^{[\alpha]}\right)<\infty$. For a group $G$ of automorphisms 
of $V$ such that $g\in G$,  
a {\it generalized  $g$-twisted 
$V$-module with a $G$-action}, a {\it lower-bounded  generalized $g$-twisted $V$-module 
with  a $G$-action},  a {\it grading-restricted generalized $g$-twisted $V$-module with a $G$-action},
and a {\it quasi-finite-dimensional  generalized $g$-twisted $V$-module
with a $G$-action} are defined by combining the definitions above 
accordingly. }
\end{defn}

We also need to define $V$-module maps between weak and generalized
$g$-twisted modules without  $g$-actions, with $g$-actions, or with $G$-actions
 for a fixed automorphism $g$ of $V$. 

\begin{defn}
{\rm For weak $g$-twisted $V$-modules  $W_{1}$ and $W_{2}$ without $g$-actions, a {\it $V$-module map from $W_{1}$ to 
$W_{2}$} is a linear map$ f: W_{1}\to W_{2}$ such that $f\circ Y_{W_{1}}(v, x)=Y_{W_{2}}(v, x)\circ f$,
$f\circ L_{W_{1}}(0)=L_{W_{2}}(0)\circ f$, and  $f\circ L_{W_{1}}(-1)=L_{W_{2}}(-1)\circ f$.
For weak $g$-twisted $V$-modules  $W_{1}$ and $W_{2}$ with $g$-actions, a {\it $V$-module map from $W_{1}$ to 
$W_{2}$} is a $V$-module map $f: W_{1}\to W_{2}$
when $W_{1}$ and $W_{2}$ are viewed as 
weak $g$-twisted $V$-modules without $g$-actions satisfying the additional condition 
$f\circ g=g\circ f$. Let $G$ be a group of automorphisms of $V$ containing $g$. For 
weak $g$-twisted $V$-modules  $W_{1}$ and $W_{2}$ with $G$-actions, a {\it $V$-module map from $W_{1}$ to 
$W_{2}$} is a $V$-module map $f: W_{1}\to W_{2}$ when $W_{1}$ and $W_{2}$ are viewed as 
weak $g$-twisted $V$-modules with $g$-actions such that $f$ is also a $G$-module map.
For generalized $g$-twisted $V$-modules  $W_{1}$ and $W_{2}$ without $g$-actions, with $g$-actions 
or with $G$-actions, 
a {\it $V$-module map from $W_{1}$ to  $W_{2}$} is a $V$-module map from $W_{1}$ to 
$W_{2}$ when $W_{1}$ and $W_{2}$ are viewed as weak $g$-twisted $V$-modules without $g$-actions, with $g$-actions 
or with $G$-actions, respectively. }
\end{defn}

In this paper, we mostly study lower-bounded, grading-restricted, quasi-finite-dimensional 
$g$-twisted $V$-modules with  and without $g$-actions, and also with $G$-actions for 
a group of automorphisms of $V$. Since by definition, generalized $g$-twisted $V$-modules 
with $g$-actions and with $G$-actions are also generalized $g$-twisted $V$-modules  without $g$-actions, 
the results obtained in this paper for suitable  generalized $g$-twisted $V$-modules  without $g$-actions 
hold also for the corresponding 
generalized $g$-twisted $V$-modules with $g$-actions and with $G$-actions when $G$ is abelian.
Since a large part of the results in the paper is about suitable $g$-twisted $V$-modules without $g$-actions,
for simplicity, we will omit ``without a $g$-action'' in the term  ``a weak $g$-twisted $V$-module without a $g$-action''
to call such a twisted $V$-module simply a weak $g$-twisted $V$-module.
Similarly, we will omit ``without a $g$-action''
in the terms ``a generalized $g$-twisted $V$-module without a $g$-action,'' ``a lower-bounded 
generalized $g$-twisted $V$-module without a $g$-action,''  ``a grading-restricted 
generalized $g$-twisted $V$-module without a $g$-action,'' ``and a quasi-finite-dimensional 
generalized $g$-twisted $V$-module.'' If a result is specifically about 
suitable $g$-twisted $V$-modules with $g$- or $G$-actions, we will always 
explicitly say so by using the words ``with $g$-actions''  or ``with $G$-actions.''

For a weak $g$-twisted $V$-module $W$, let 
$$Y_{W, 0}(v, x)=\sum_{n\in \C}v_{n, 0}x^{-n-1}$$
for $v\in V$. Then by Lemma 2.3 in \cite{HY}, we have 
\begin{equation}\label{HY-lem-2.3}
Y_{W}(v, x)=Y_{W, 0}(x^{-\mathcal{N}_{g}}v, x)
\end{equation}
for $v\in V$. For simplicity, we shall also denote $v_{n, 0}$ sometimes simply by $v_{n}$ for $v\in V$ and $n\in \C$. 

For a generalized  $g$-twisted 
$V$-module 
$$W = \coprod_{n \in \C}\coprod_{\alpha\in \C/\Z} W_{[n]}^{[\alpha]}$$
with a $g$-action, let $W^{[\alpha]}=\coprod_{n\in \C}W_{[n]}^{[\alpha]}$ for $\alpha\in \C/\Z$.
If $W^{[\alpha]}\ne 0$ for $\alpha\in \C/\Z$, then $W^{[\alpha]}$ is the generalized eigenspace 
of the action of $g$ on $W$ with eigenvalue $e^{2\pi \i \alpha}$. For  $\alpha\in \C/\Z$,
we can always find a unique complex number in $\alpha$ 
such that the real part of this complex number is in $[0, 1)$. By abusing the notation, we  denote 
this number by $\alpha$ and will always use $\alpha, \beta, \gamma, \dots$ to denote such unique
representatives of elements of $\C/\Z$. For $\alpha\in \C$ satisfying $\Re(\alpha)\in [0, 1)$, 
We use $W^{[\alpha]}$ to denote the space $W^{[\alpha+\Z]}$. 
Let 
$$P_{W}^{g}=\{\alpha\in \C\mid \Re(\alpha)\in [0, 1),\; W^{[\alpha]}\ne 0\}.$$
Then we have 
$$W = \coprod_{n \in \C}\coprod_{\alpha\in P_{W}^{g}} W_{[n]}^{[\alpha]}$$

General twisted intertwining operators, especially those among twisted modules twisted 
by noncommuting automorphisms,  were introduced and studied in \cite{H-twisted-int} and \cite{DH}. 
In the noncommuting case studied in these works, one must use the complex analytic approach. 
On the other hand, in the case of twisted intertwining operators among 
twisted modules twisted 
by commuting automorphisms, twisted intertwining operators can be 
defined and studied using a Jacobi identity formulated using the formal variable approach. 
In this paper, we study only the case of commuting automorphisms and we need the Jacobi identity.
So here we give a formal variable definition of twisted intertwining operators.  
This formal variable definition can be derived from the analytic definition 
in  \cite{H-twisted-int} and \cite{DH}.  We will discuss the connection between the 
complex analytic approach and the formal variable approach in another paper.

\begin{defn}
{\rm Let $V$ be a grading-restricted vertex algebra, $g_{1}$ and $g_{2}$ commuting automorphisms of $V$,
and $W_{1}$, $W_{2}$, $W_{3}$ weak $g_{1}$-, $g_{2}$, $g_{1}g_{2}$-twisted $V$-modules,
respectively.  
A {\it twisted intertwining operator of type $\binom{W_{3}}{W_{1}W_{2}}$} is a linear map
\begin{align*}
\Y: W_{1}\otimes W_{2}&\to W_{3}\{x\}[\log x]\nn
w_{1}\otimes w_{2}&\mapsto \Y(w_{1}, x)w_{2}=\sum_{k=0}^{K}\sum_{n\in \C}\Y_{n, k}(w_{1})w_{2}x^{-n-1}(\log x)^{k}
\end{align*}
satisfying the following axioms:
\begin{enumerate}

\item The {\it lower-truncation property}: For $w_{1}\in W_{1}$, $w_{2}\in W_{2}$,
and $k=0, \dots, K$, there exists $N\in \N$ such that $\Y_{n, k}(w_{1})w_{2}=0$ when $\Re(n)> N$.

\item The {\it Jacobi identity}: For $v\in V$, $w_{1}\in W_{1}$, and $w_{2}\in W_{2}$, 
\begin{align}\label{jacobi-twisted-int-0}
&x_{0}^{-1}\delta\left(\frac{x_{1}-x_{2}}{x_{0}}\right)
Y_{W_{3}}\left(\left(\frac{x_{1}-x_{2}}{x_{0}}\right)^{\mathcal{L}_{g_{1}}}v, x_{1}\right)\Y(w_{1}, x_{2})\nn
&\quad\quad - 
x_{0}^{-1}\delta\left(\frac{x_{2}-x_{1}}{-x_{0}}\right)
\Y(w_{1}, x_{2})Y_{W_{2}}\left(e^{\pi \i\mathcal{L}_{g_{1}}}\left(\frac{x_{2}-x_{1}}{x_{0}}\right)^{\mathcal{L}_{g_{1}}}
v, x_{1}\right)\nn
&\quad=
x_{1}^{-1}\delta\left(\frac{x_{2}+x_{0}}{x_{1}}\right)
\Y\left(Y_{W_{1}}\left(\left(\frac{x_{2}+x_{0}}{x_{1}}\right)^{\mathcal{L}_{g_{2}}}v, x_{0}\right)w_{1}, x_{2}\right).
\end{align}

\item The {\it $L(0)$-commutator formula}: For $w_{1}\in W_{1}$, 
\begin{align*}
L_{W_{3}}(0)\Y(w_{1}, x)-\Y(w_{1}, x)L_{W_{2}}(0)=x\frac{d}{dx}\Y(w_{1}, x)+\Y(L_{W_{1}}(0)w_{1}, x).
\end{align*}

\item The {\it $L(-1)$-derivative property}: For $w_{1}\in W_{1}$, 
\begin{align*}
\frac{d}{dx}\Y(w_{1}, x)=\Y(L_{W_{1}}(-1)w_{1}, x)=L_{W_{3}}(-1)\Y(w_{1}, x)-\Y(w_{1}, x)L_{W_{2}}(-1).
\end{align*}
\end{enumerate}
In the case that $W_{1}$, $W_{2}$, $W_{3}$ are weak $g_{1}$-, $g_{2}$, $g_{1}g_{2}$-twisted $V$-modules
with $G$-actions when $G$ is abelian and 
contains $g_{1}$ and $g_{2}$, 
a {\it twisted intertwining operator of  type $\binom{W_{3}}{W_{1}W_{2}}$} is a 
twisted intertwining operator of  type $\binom{W_{3}}{W_{1}W_{2}}$ when 
$W_{1}$, $W_{2}$, $W_{3}$ are viewed as weak $g_{1}$-, $g_{2}$, $g_{1}g_{2}$-twisted $V$-modules
(without $g_{1}$-, $g_{2}$, $g_{1}g_{2}$-actions) satisfying the additional 
condition 
$$h\Y(w_{1}, x)w_{2}=\Y(hw_{1}, x)hw_{2}$$
for $h\in G$, $w_{1}\in W_{1}$ and $w_{2}\in W_{2}$.
For  generalized $g_{1}$-, $g_{2}$, $g_{1}g_{2}$-twisted $V$-modules $W_{1}$, $W_{2}$, and $W_{3}$,
respectively, 
a  {\it twisted intertwining operator of type $\binom{W_{3}}{W_{1}W_{2}}$} is a 
twisted intertwining operator of the same type when $W_{1}$, $W_{2}$, and $W_{3}$ are viewed as 
weak $g_{1}$-, $g_{2}$, $g_{1}g_{2}$-twisted $V$-modules, respectively. 
For  generalized $g_{1}$-, $g_{2}$, $g_{1}g_{2}$-twisted $V$-modules $W_{1}$, $W_{2}$, and $W_{3}$
with $G$-actions, respectively,
a  {\it twisted intertwining operator of type $\binom{W_{3}}{W_{1}W_{2}}$} is a 
twisted intertwining operator of the same type when $W_{1}$, $W_{2}$, and $W_{3}$ are viewed as 
weak $g_{1}$-, $g_{2}$, $g_{1}g_{2}$-twisted $V$-modules with $G$-actions, respectively. }
\end{defn}

\begin{rema}
{\rm In the second term in the Jacobi identity \eqref{jacobi-twisted-int-0}, 
we choose the value of $(-1)^{\mathcal{L}_{g_{1}}}$
to be $e^{\pi \i\mathcal{L}_{g_{1}}}$. We can also choose the value of $(-1)^{\mathcal{L}_{g_{1}}}$
to be $e^{-\pi \i\mathcal{L}_{g_{1}}}$. Different choices of the values 
corresponding to different choices of single-valued branches of the multivalued analytic function 
to which the products and iterates of twisted vertex operators and the intertwining operator converge. 
These choices are related to the crossed braiding of the tensor category to be constructed. 
See \cite{T3} for a discussion of both choices in the case that $g_{1}$ and $g_{2}$ are of finite orders.}
\end{rema}

The definition of twisted intertwining operator 
above is given using formal variables. We will need the evaluation of the 
formal series $\Y(w_{1}, x)$ for $w_{1}\in W_{1}$ at $z\in \C^{\times}$ in later sections. Since $\Y(w_{1}, x)$ 
contains nonintegral powers and the logarithm of $x$, the evaluation depends on 
a choice of the values of the logarithm of $z$. Let $l_{p}(z)=\log |z|+\i \arg z+2\pi \i p$,
where $0\le \arg z<2\pi$. We also use the notation $\log z=l_{0}(z)$. 
For $w_{1}\in W_{1}$ and $w_{2}\in W_{2}$, we have
$$\Y(w_{1}, x)w_{2}=\sum_{k=0}^{K}\sum_{n\in \C}\Y_{n, k}(w_{1})w_{2}x^{-n-1}(\log x)^{k}.$$
Then for $p\in \Z$, we define the $p$-th value $\Y^{p}(w_{1}, z)w_{2}$ of $\Y(w_{1}, x)w_{2}$ at $x=z$ to be 
$$\Y^{p}(w_{1}, z)w_{2}=\sum_{k=0}^{K}\sum_{n\in \C}\Y_{n, k}(w_{1})w_{2}e^{(-n-1)l_{p}(z)}l_{p}(z)^{k}
\in \overline{W}_{3}=\prod_{n\in \C}(W_{3})_{[n]}.$$
In the case $p=0$, we denote $\Y^{0}(w_{1}, w_{1})w_{2}$ simply by 
$\Y(w_{1}, z)w_{2}$. 

We now derive a version of the Jacobi identity expressed in terms of the components of 
the twisted vertex operators. We need this version in later sections. Recall that we use $v_{n}$ to denote 
$v_{n, 0}$ for $v\in V$ and $n\in \C$. 

\begin{prop}
Let $V$ be a grading-restricted vertex algebra, $g_{1}$ and $g_{2}$ commuting automorphisms of $V$,
$W_{1}$, $W_{2}$, $W_{3}$ weak $g_{1}$-, $g_{2}$-, $g_{1}g_{2}$-twisted $V$-modules and 
$\Y$  a twisted intertwining operator of type $\binom{W_{3}}{W_{1}W_{2}}$. Then for $\alpha_{1}\in P_{V}^{g_{1}}$, 
$\alpha_{2}\in P_{V}^{g_{2}}$, $v\in V^{[\alpha_{1}, \alpha_{2}]}$, $w_{1}\in W_{1}$, 
\begin{align}\label{jacobi-twisted-int-2}
&\sum_{k\in \N}\sum_{l\in \N}\binom{\alpha_{1}-n_{1}}{k}(-x)^{k+l}
\left(\binom{\mathcal{N}_{g_{1}}}{l}v\right)_{\alpha_{1}+\alpha_{2}-n_{1}-n_{2}-k-l}\Y(w_{1}, x)\nn
&\quad \quad-\sum_{k\in \N}\sum_{l\in \N}\binom{\alpha_{1}-n_{1}}{k}
e^{\pi \i (\alpha_{1}-n_{1}-k-l)}x^{\alpha_{1}-n_{1}-k-l}
\Y(w_{1}, x)\left(e^{\pi \i\mathcal{N}_{g_{1}}}\binom{\mathcal{N}_{g_{1}}}{l}
x^{\mathcal{N}_{g_{1}}}v\right)_{\alpha_{2}-n_{2}+k+l}\nn
&\quad=\sum_{k\in \N}\sum_{l\in \N}\binom{\alpha_{2}-n_{2}}{k}
x^{\alpha_{2}-n_{2}-k-l}
\Y\left(\left(\binom{\mathcal{N}_{g_{2}}}{l}x^{\mathcal{N}_{g_{2}}}v\right)_{\alpha_{1}-n_{1}+k+l}w_{1}, x\right).
\end{align}
\end{prop}
\begin{proof}
For $\alpha_{1}\in P_{V}^{g_{1}}$
and $\alpha_{2}\in P_{V}^{g_{2}}$, let $V^{[\alpha_{1}, \alpha_{2}]}=V^{]\alpha_{1}]}\cap V^{[\alpha_{2}]}$.
Then $V^{[\alpha_{1}, \alpha_{2}]}$ the space of common eigenvectors of $g_{1}$ and $g_{2}$ with eigenvalues 
$e^{2\pi \i\alpha_{1}}$ and $e^{2\pi \i\alpha_{2}}$, respectively, and we have
$V=\coprod_{\alpha_{1}\in P_{V}^{g_{1}}, \alpha_{2}\in P_{V}^{g_{2}}}V^{[\alpha_{1}, \alpha_{2}]}$.
For $v\in V^{[\alpha_{1}, \alpha_{2}]}_{g_{1}, g_{2}}$ and $w_{1}\in W_{1}$, 
we have $\mathcal{S}_{g_{1}}v=\alpha_{1}v$ and $\mathcal{S}_{g_{2}}v=\alpha_{2}v$.
Then 
\begin{align*}
\left(\frac{x_{1}-x_{2}}{x_{0}}\right)^{\mathcal{L}_{g_{1}}}v&=\left(\frac{x_{1}-x_{2}}{x_{0}}\right)^{\alpha_{1}}
\left(\frac{x_{1}-x_{2}}{x_{0}}\right)^{\mathcal{N}_{g_{1}}}v,\\
e^{\pi \i\mathcal{L}_{g_{1}}}\left(\frac{x_{2}-x_{1}}{x_{0}}\right)^{\mathcal{L}_{g_{1}}}
v&=e^{\pi \i\alpha_{1}}\left(\frac{x_{2}-x_{1}}{x_{0}}\right)^{\alpha_{1}}
e^{\pi \i\mathcal{N}_{g_{1}}}\left(\frac{x_{2}-x_{1}}{x_{0}}\right)^{\mathcal{N}_{g_{1}}}v,\\
\left(\frac{x_{2}+x_{0}}{x_{1}}\right)^{\mathcal{L}_{g_{2}}}v
&=\left(\frac{x_{2}+x_{0}}{x_{1}}\right)^{\alpha_{2}}\left(\frac{x_{2}+x_{0}}{x_{1}}\right)^{\mathcal{N}_{g_{2}}}v.
\end{align*}
Using these formulas, we see that for $v\in V^{[\alpha_{1}, \alpha_{2}]}_{g_{1}, g_{2}}$ and $w_{1}\in W_{1}$, 
the Jacobi identity \eqref{jacobi-twisted-int-0} becomes
\begin{align}\label{jacobi-twisted-int-0.1}
&x_{0}^{-1}\delta\left(\frac{x_{1}-x_{2}}{x_{0}}\right)\left(\frac{x_{1}-x_{2}}{x_{0}}\right)^{\alpha_{1}}
Y_{W_{3}}\left(\left(\frac{x_{1}-x_{2}}{x_{0}}\right)^{\mathcal{N}_{g_{1}}}v, x_{1}\right)\Y(w_{1}, x_{2})\nn
&\quad\quad - 
x_{0}^{-1}\delta\left(\frac{x_{2}-x_{1}}{-x_{0}}\right)e^{\pi \i\alpha_{1}}\left(\frac{x_{2}-x_{1}}{x_{0}}\right)^{\alpha_{1}}
\Y(w_{1}, x_{2})Y_{W_{2}}\left(e^{\pi \i\mathcal{N}_{g_{1}}}\left(\frac{x_{2}-x_{1}}{x_{0}}\right)^{\mathcal{N}_{g_{1}}}
v, x_{1}\right)\nn
&\quad=x_{1}^{-1}\delta\left(\frac{x_{2}+x_{0}}{x_{1}}\right)\left(\frac{x_{2}+x_{0}}{x_{1}}\right)^{\alpha_{2}}
\Y\left(Y_{W_{1}}\left(\left(\frac{x_{2}+x_{0}}{x_{1}}\right)^{\mathcal{N}_{g_{2}}}v, x_{0}\right)w_{1}, x_{2}\right).
\end{align}
By \eqref{HY-lem-2.3}, we have 
\begin{align*}
Y_{W_{3}}\left(\left(\frac{x_{1}-x_{2}}{x_{0}}\right)^{\mathcal{N}_{g_{1}}}v, x_{1}\right)
&=Y_{W_{3}; 0}((1-x_{1}^{-1}x_{2})^{\mathcal{N}_{g_{1}}}x_{0}^{-\mathcal{N}_{g_{1}}}
x_{1}^{-\mathcal{N}_{g_{2}}}v, x_{1}),\\
Y_{W_{2}}\left(e^{\pi \i\mathcal{N}_{g_{1}}}\left(\frac{x_{2}-x_{1}}{x_{0}}\right)^{\mathcal{N}_{g_{1}}}
v, x_{1}\right)
&=Y_{W_{2}, 0}(e^{\pi \i\mathcal{N}_{g_{1}}}(1-x_{1}x_{2}^{-1})^{\mathcal{N}_{g_{1}}}
x_{0}^{-\mathcal{N}_{g_{1}}}
x_{1}^{-\mathcal{N}_{g_{2}}}x_{2}^{\mathcal{N}_{g_{1}}}v, x_{1}),\\
Y_{W_{1}}\left(\left(\frac{x_{2}+x_{0}}{x_{1}}\right)^{\mathcal{N}_{g_{2}}}v, x_{0}\right)
&=Y_{W_{1}, 0}((1+x_{0}x_{2}^{-1})^{\mathcal{N}_{g_{2}}}x_{0}^{-\mathcal{N}_{g_{1}}}x_{1}^{-\mathcal{N}_{g_{2}}}
x_{2}^{\mathcal{N}_{g_{2}}}v, x_{0})
\end{align*}
Using these formulas, and replacing $v$ by $x_{0}^{\mathcal{N}_{g_{1}}}x_{1}^{\mathcal{N}_{g_{2}}}v$,
we see that \eqref{jacobi-twisted-int-0.1} becomes 
\begin{align}\label{jacobi-twisted-int-0.2}
&x_{0}^{-1}\delta\left(\frac{x_{1}-x_{2}}{x_{0}}\right)\left(\frac{x_{1}-x_{2}}{x_{0}}\right)^{\alpha_{1}}
Y_{W_{3}; 0}((1-x_{1}^{-1}x_{2})^{\mathcal{N}_{g_{1}}}
v, x_{1})\Y(w_{1}, x_{2})\nn
&\quad\quad - 
x_{0}^{-1}\delta\left(\frac{x_{2}-x_{1}}{-x_{0}}\right)e^{\pi \i\alpha_{1}}\left(\frac{x_{2}-x_{1}}{x_{0}}\right)^{\alpha_{1}}
\Y(w_{1}, x_{2})Y_{W_{2}, 0}(e^{\pi \i\mathcal{N}_{g_{1}}}(1-x_{1}x_{2}^{-1})^{\mathcal{N}_{g_{1}}}
x_{2}^{\mathcal{N}_{g_{1}}}v, x_{1})\nn
&\quad=x_{1}^{-1}\delta\left(\frac{x_{2}+x_{0}}{x_{1}}\right)\left(\frac{x_{2}+x_{0}}{x_{1}}\right)^{\alpha_{2}}
\Y(Y_{W_{1}, 0}((1+x_{0}x_{2}^{-1})^{\mathcal{N}_{g_{2}}}x_{2}^{\mathcal{N}_{g_{2}}}v, x_{0})
w_{1}, x_{2}).
\end{align}
Multiplying $x_{0}^{\alpha_{1}}x_{1}^{\alpha_{2}}$ to both sides of \eqref{jacobi-twisted-int-0.1},
we obtain the following version of the Jacobi identity:
\begin{align}\label{jacobi-twisted-int}
&x_{0}^{-1}\delta\left(\frac{x_{1}-x_{2}}{x_{0}}\right)(x_{1}-x_{2})^{\alpha_{1}}x_{1}^{\alpha_{2}}
Y_{W_{3}; 0}((1-x_{1}^{-1}x_{2})^{\mathcal{N}_{g_{1}}}v, x_{1})
\Y(w_{1}, x_{2})\nn
&\quad\quad - 
x_{0}^{-1}\delta\left(\frac{x_{2}-x_{1}}{-x_{0}}\right)e^{\pi \i \alpha_{1}}(x_{2}-x_{1})^{\alpha_{1}}x_{1}^{\alpha_{2}}
\Y(w_{1}, x_{2})Y_{W_{2}, 0}(e^{\pi \i\mathcal{N}_{g_{1}}}(1-x_{1}x_{2}^{-1})^{\mathcal{N}_{g_{1}}}
x_{2}^{\mathcal{N}_{g_{1}}}v, x_{1})\nn
&\quad=
x_{1}^{-1}\delta\left(\frac{x_{2}+x_{0}}{x_{1}}\right)x_{0}^{\alpha_{1}}(x_{2}+x_{0})^{\alpha_{2}}
\Y(Y_{W_{1}, 0}((1+x_{0}x_{2}^{-1})^{\mathcal{N}_{g_{2}}}x_{2}^{\mathcal{N}_{g_{2}}}v, x_{0})w_{1}, x_{2}).
\end{align}

Taking $\res_{x_{0}}\res_{x_{1}}x_{0}^{-n_{1}}x_{1}^{-n_{2}}$ 
of the first term in the left-hand side of the Jacobi identity \eqref{jacobi-twisted-int}, we obtain
\begin{align*}
&\res_{x_{0}}\res_{x_{1}}x_{0}^{-n_{1}}x_{1}^{-n_{2}}
x_{0}^{-1}\delta\left(\frac{x_{1}-x_{2}}{x_{0}}\right)(x_{1}-x_{2})^{\alpha_{1}}x_{1}^{\alpha_{2}}
Y_{W_{3}; 0}((1-x_{1}^{-1}x_{2})^{\mathcal{N}_{g_{1}}}
v, x_{1})\Y(w_{1}, x_{2})\nn
&\quad =\res_{x_{1}}x_{1}^{-n_{2}}
(x_{1}-x_{2})^{\alpha_{1}-n_{1}}x_{1}^{\alpha_{2}}
Y_{W_{3}; 0}((1-x_{1}^{-1}x_{2})^{\mathcal{N}_{g_{1}}}
v, x_{1})\Y(w_{1}, x_{2})\nn
&\quad =\sum_{k\in \N}\sum_{l\in \N}\binom{\alpha_{1}-n_{1}}{k}
\res_{x_{1}}x_{1}^{\alpha_{1}+\alpha_{2}-n_{1}-n_{2}-k-l}(-x_{2})^{k+l}
Y_{W_{3}; 0}\left(\binom{\mathcal{N}_{g_{1}}}{l}v, x_{1}\right)\Y(w_{1}, x_{2})\nn
&\quad=\sum_{k\in \N}\sum_{l\in \N}\binom{\alpha_{1}-n_{1}}{k}(-x_{2})^{k+l}
\left(\binom{\mathcal{N}_{g_{1}}}{l}v\right)_{\alpha_{1}+\alpha_{2}-n_{1}-n_{2}-k-l}\Y(w_{1}, x_{2}).
\end{align*}
Taking $\res_{x_{0}}\res_{x_{1}}x_{0}^{-n_{1}}x_{1}^{-n_{2}}$ 
of the second term in the left-hand side of the Jacobi identity \eqref{jacobi-twisted-int}, we obtain
\begin{align*}
&-\res_{x_{0}}\res_{x_{1}}x_{0}^{-n_{1}}x_{1}^{-n_{2}}x_{0}^{-1}\delta\left(\frac{x_{2}-x_{1}}{-x_{0}}\right)
e^{\pi \i \alpha_{1}}(x_{2}-x_{1})^{\alpha_{1}}
x_{1}^{\alpha_{2}}\cdot\nn
&\quad\quad\quad\quad\quad\quad
\cdot \Y(w_{1}, x_{2})Y_{W_{2}, 0}(e^{\pi \i\mathcal{N}_{g_{1}}}(1-x_{1}x_{2}^{-1})^{\mathcal{N}_{g_{1}}}
x_{2}^{\mathcal{N}_{g_{1}}}v, x_{1})\nn
&\quad=-\res_{x_{1}}x_{1}^{-n_{2}}
e^{\pi \i (\alpha_{1}-n_{1})}(x_{2}-x_{1})^{\alpha_{1}-n_{1}}
x_{1}^{\alpha_{2}}\cdot\nn
& \quad\quad\quad\quad\quad\quad
\cdot \Y(w_{1}, x_{2})Y_{W_{2}, 0}(e^{\pi \i\mathcal{N}_{g_{1}}}(1-x_{1}x_{2}^{-1})^{\mathcal{N}_{g_{1}}}
x_{2}^{\mathcal{N}_{g_{1}}}v, x_{1})\nn
&\quad=-\sum_{k\in \N}\sum_{l\in \N}\binom{\alpha_{1}-n_{1}}{k}\res_{x_{1}}
e^{\pi \i (\alpha_{1}-n_{1})}x_{2}^{\alpha_{1}-n_{1}-k-l}
(-1)^{k+l}x_{1}^{\alpha_{2}-n_{2}+k+l}\cdot\nn
&\quad\quad\quad\quad\quad\quad
\cdot \Y(w_{1}, x_{2})Y_{W_{2}, 0}\left(e^{\pi \i\mathcal{N}_{g_{1}}}\binom{\mathcal{N}_{g_{1}}}{l}
x_{2}^{\mathcal{N}_{g_{1}}}v, x_{1}\right)\nn
&\quad=-\sum_{k\in \N}\sum_{l\in \N}\binom{\alpha_{1}-n_{1}}{k}
e^{\pi \i (\alpha_{1}-n_{1}-k-l)}x_{2}^{\alpha_{1}-n_{1}-k-l}\cdot\nn
&\quad\quad\quad\quad\quad\quad
\cdot
\Y(w_{1}, x_{2})\left(e^{\pi \i\mathcal{N}_{g_{1}}}\binom{\mathcal{N}_{g_{1}}}{l}
x_{2}^{\mathcal{N}_{g_{1}}}v\right)_{\alpha_{2}-n_{2}+k+l}.
\end{align*}
Taking $\res_{x_{0}}\res_{x_{1}}x_{0}^{-n_{1}}x_{1}^{-n_{2}}$ 
of the right-hand side of the Jacobi identity \eqref{jacobi-twisted-int}, we obtain
\begin{align*}
&\res_{x_{0}}\res_{x_{1}}x_{0}^{-n_{1}}x_{1}^{-n_{2}}
x_{1}^{-1}\delta\left(\frac{x_{2}+x_{0}}{x_{1}}\right)x_{0}^{\alpha_{1}}(x_{2}+x_{0})^{\alpha_{2}}\cdot\nn
&\quad\quad\quad\quad\quad\quad\cdot
\Y(Y_{W_{1}, 0}((1+x_{0}x_{2}^{-1})^{\mathcal{N}_{g_{2}}}x_{2}^{\mathcal{N}_{g_{2}}}v, x_{0})w_{1}, x_{2})\nn
&\quad=\res_{x_{0}}x_{0}^{\alpha_{1}-n_{1}}(x_{2}+x_{0})^{\alpha_{2}-n_{2}}
\Y(Y_{W_{1}, 0}((1+x_{0}x_{2}^{-1})^{\mathcal{N}_{g_{2}}}x_{2}^{\mathcal{N}_{g_{2}}}v, x_{0})w_{1}, x_{2})\nn
&\quad=\sum_{k\in \N}\sum_{l\in \N}\binom{\alpha_{2}-n_{2}}{k}\res_{x_{0}}x_{0}^{\alpha_{1}-n_{1}+k+l}
x_{2}^{\alpha_{2}-n_{2}-k-l}
\Y\left(Y_{W_{1}, 0}\left(\binom{\mathcal{N}_{g_{2}}}{l}x_{2}^{\mathcal{N}_{g_{2}}}v, x_{0}\right)w_{1}, x_{2}\right)\nn
&\quad=\sum_{k\in \N}\sum_{l\in \N}\binom{\alpha_{2}-n_{2}}{k}
x_{2}^{\alpha_{2}-n_{2}-k-l}
\Y\left(\left(\binom{\mathcal{N}_{g_{2}}}{l}x_{2}^{\mathcal{N}_{g_{2}}}v\right)_{\alpha_{1}-n_{1}+k+l}w_{1}, x_{2}\right).
\end{align*}
Using these calculations and \eqref{jacobi-twisted-int}, and changing $x_{2}$ to $x$, we obtain
 \eqref{jacobi-twisted-int-2}.
\end{proof}

\begin{rema}
{\rm In the special case that $g_{1}$ and $g_{2}$ acts semisimply on $V$, 
 \eqref{jacobi-twisted-int-2} becomes
\begin{align*}\label{jacobi-twisted-int-3}
&\sum_{k\in \N}\binom{\alpha_{1}-n_{1}}{k}(-x)^{k}v_{\alpha_{1}+\alpha_{2}-n_{1}-n_{2}-k}\Y(w_{1}, x)\nn
&\quad \quad -\sum_{k\in \N}\binom{\alpha_{1}-n_{1}}{k}e^{\pi \i (\alpha_{1}-n_{1}-k)}
x^{\alpha_{1}-n_{1}-k}
\Y(w_{1}, x)v_{\alpha_{2}-n_{2}+k}\nn
&\quad=\sum_{k\in \N}\binom{\alpha_{2}-n_{2}}{k}x^{\alpha_{2}-n_{2}-k}
\Y(v_{\alpha_{1}-n_{1}+k}w_{1}, x).
\end{align*}}
\end{rema}

\setcounter{equation}{0}

\section{Cofiniteness of twisted modules}

In this section, we prove some basic properties of generalized twisted $V$-modules. 

\begin{defn}
{\rm Let $W$ be a weak $g$-twisted $V$-module. For each $n \in 2+\N$, we define $C_n(W)$ to be
the subspace of $W$ spanned by the elements of the form 
$u_{\alpha-n}w$ for  $u\in V^{[\alpha]}_{+}$, $\alpha \in P_{V}^{g}$, $w \in W$. 
We say that $W$ is {\it $C_n$-cofinite} if $\dim W/C_n(W) < \infty$. }
\end{defn}

As in the case of (untwisted) weak modules, we also have the following useful fact:

\begin{prop}
Let $W$ be a weak $g$-twisted $V$-module. If $W$ is $C_{n}$-cofinite for some $n\in 2+\N$, 
then it is also $C_{m}$-cofinite for $m=2, \dots, n$. 
\end{prop}
\begin{proof}
As in the untwisted case, this result follows from the $L(-1)$-derivative property. 
\end{proof}

A $C_{n}$-cofinite lower-bounded
generalized $g$-twisted  $V$-module has the following properties:

\begin{prop}\label{C-2-gr-wk}
Let $W= \coprod_{m \in \C}W_{[n]}$
be a $C_{n}$-cofinite lower-bounded generalized
$g$-twisted $V$-module. Then $W$ has the following properties:
\begin{enumerate}
\item $W$ is quasi-finite-dimensional. 

\item There exists a finite-dimensional subspace $M$ of $W$
such that $W$ is spanned by elements of the form 
\begin{equation}\label{C-2-gr-wk-0.1}
v^{(1)}_{\alpha_{1}-n}
\cdots v^{(i)}_{\alpha_{i}-n}w
\end{equation}
for $i\in \N$,
$v^{(1)}\in V^{[\alpha_{1}]}_{+}, \dots, v^{(i)}\in V^{[\alpha_{i}]}_{+}$ and $w\in M$. 
In particular, $W$ is finitely generated. 

\item  If $W$ has an action of another automorphism $h$ of $V$ such that the $h$-action commutes
with $L_{W}(0)$, then $W$ can be 
decomposed as a direct sum of generalized eigenspaces of the action of $h$. 
\end{enumerate}
\end{prop}
\begin{proof}
To prove Property 1, we need to prove that 
$$\dim \left(\coprod_{\Re(m)\le N} W_{[m]}\right)<\infty$$
for $N\in \Z$. 
Since $W$ is lower bounded, there exists $N_{0}\in \Z$ such that $W_{[m]}=0$ when $\Re(m)<N_{0}$. 
We use induction on $N-N_{0}$. In the case $N-N_{0}=0$, 
for $v\in V_{(l)}^{[\alpha]}$ and $w\in W_{[m]}$, where $l\in \Z_{+}$ and $m\in \C$ satisfying $\Re(m)\ge N_{0}$,
$v_{\alpha-n}w\in W_{[l-\alpha+n-1+m]}$. Since $l\in \Z_{+}$, $n\in 2+\N$, and $\Re(m)\ge N_{0}$, 
$\Re(l-\alpha+n-1+m)> N_{0}$. So nonzero elements of $W_{[l-\alpha+n-1+m]}$ such as 
$v_{\alpha-n}w$ cannot be in $\coprod_{\Re(m)=N_{0}}W_{[m]}$.
 Therefore $\dim \coprod_{\Re(m)=N_{0}}W_{[m]}
\le \dim W/C_{n}(W)<\infty$ since $W$ is $C_{n}$-cofinite. This proves Property 1 in the case $N-N_{0}=0$.

Assume that Property 1 is true for $0\le N-N_{0}\le k$. Now let $N=N_{0}+k+1$. 
By the induction assumption, we need only 
prove 
\begin{equation}\label{C-2-gr-wk-0.5}
\dim \left(\coprod_{k<\Re(m)-N_{0}\le k+1}W_{[m]}\right)
\end{equation}
For $m\in \C$ satisfying $k<\Re(m)-N_{0}\le k+1$, let 
$(C_{n}(W)\cap W_{[m]})_{\alpha, l}$ be the subspace 
of $C_{n}(W)\cap W_{[m]}$ spanned by elements of the form 
$v_{\alpha-n}w$ for $v\in V_{(m+\alpha-n+1-l)}^{[\alpha]}$ and $w\in W_{[l]}$
for $l\in \C$ satisfying $0\le \Re(l)-N_{0}\le k$ and $m+\alpha-n+1-l\in \Z_{+}$. 
For $m, l\in \C$, $\alpha\in P_{V}^{g}$ such that $m+\alpha-n+1-l\in \Z_{+}$, or equivalently,
$m+\alpha-n-l\in \N$, 
we have $\Im(l)=\Im(m)+\Im(\alpha)$ and $\Re(l)\in \Re(m)+\Re(\alpha)-n+1- \Z_{+}
=\Re(m)+\Re(\alpha)-n- \N$. Let $p=m+\alpha-n-l\in \N$. 
If $m$ and $l$ further satisfy $0\le \Re(l)-N_{0}\le k<\Re(m)-N_{0}\le k+1$, then 
we have 
$$p=m+\alpha-n-l=\Re(m)+\Re(\alpha)-n-\Re(l)\le N_{0}+k+1+\Re(\alpha)-n-N_{0}<k+1-n.$$
So we obtain $0\le p\le k-n$. In particular, in the case $k<n$, there does not exist 
such $p$, or equivalently, $(C_{n}(W)\cap W_{[m]})_{\alpha, l}=0$. In the case $k\ge n$, 
for $p\in \N$ satisfying $0\le p\le  k-n$ and $l=m+\alpha-n-p$, 
we have 
\begin{align*}
k\ge \Re(m)-N_{0}-1&=\Re(l)-\Re(\alpha)+n+p-N_{0}-1\nn
&> \Re(l)-N_{0}\nn
&=\Re(m)+\Re(\alpha)-n-p-N_{0}\nn
&\ge \Re(m)-N_{0}+\Re(\alpha)-k\nn
&>k.
\end{align*}
So summing over $l\in m+\alpha-n+1-\Z_{+}$ satisfying $0\le \Re(l)-N_{0}\le k$ 
is the same as summing over $p=0, \dots, k-n$. 
Thus in the case $k\ge n$, , we have 
\begin{align}\label{C-2-gr-wk-1}
&\dim \left(C_{n}(W)\cap  \left(\coprod_{k<\Re(m)-N_{0}\le k+1} W_{[m]}\right)\right)\nn
&=\dim \left(\coprod_{k<\Re(m)-N_{0}\le k+1} C_{n}(W)\cap W_{[m]}\right)\nn
&=\dim \left(\coprod_{k<\Re(m)-N_{0}\le k+1}\coprod_{\alpha\in P_{V}^{g}}
\coprod_{\substack{ l\in m+\alpha-n+1-\Z_{+}\\0\le \Re(l)-N_{0}\le k}}
(C_{n}(W)\cap W_{[l]})_{\alpha, l}\right)\nn
&=\sum_{k<\Re(m)-N_{0}\le k+1}\sum_{\alpha\in P_{V}^{g}}
\sum_{\substack{ l\in m+\alpha-n+1-\Z_{+}\\0\le \Re(l)-N_{0}\le k}}\dim (C_{n}(W)\cap W_{[m]})_{\alpha, l}\nn
&\quad \le \sum_{k<\Re(m)-N_{0} \le k+1}\sum_{\alpha\in P_{V}^{g}}
\sum_{\substack{ l\in m+\alpha-n+1-\Z_{+}\\0\le \Re(l)-N_{0}\le k}}
\dim V_{(m+\alpha-n+1-l)}^{[\alpha]}\dim W_{[l]}\nn
&\quad \le \sum_{k<\Re(m)-N_{0} \le k+1}\sum_{\alpha\in P_{V}^{g}}\sum_{p=0}^{k-n}
\dim V_{(p)}^{[\alpha]}\dim W_{[m+\alpha-n-p]}\nn
&\quad \le \sum_{\alpha\in P_{V}^{g}}\sum_{p=0}^{k-n}
\dim V_{(p)}^{[\alpha]}\left(\sum_{k<\Re(m)-N_{0} \le k+1} \dim W_{[m+\alpha-n-p]}\right)\nn
&\quad = \sum_{p=0}^{k-n}\left(\sum_{\alpha\in P_{V}^{g}}
\dim V_{(p)}^{[\alpha]}\right)\dim \left(\coprod_{k<\Re(m)-N_{0} \le k+1} W_{[m+\alpha-n-p]}\right)\nn
&\quad \le \sum_{p=0}^{k-n}\left(\sum_{\alpha\in P_{V}^{g}}
\dim V_{(p)}^{[\alpha]}\right)\dim \left(\coprod_{0\le \Re(s)-N_{0}\le k}W_{[s]}\right)\nn
&\quad = \sum_{p=0}^{k-n}
\dim V_{(p)}\dim \left(\coprod_{0\le \Re(s)-N_{0}\le k}W_{[s]}\right)\nn
&\quad<\infty,
\end{align}
where the last step follows from the grading-restriction property of $V$
and  the induction assumption. 

On the other hand,   the quotient 
$$\left(\coprod_{k<\Re(m)-N_{0}\le k+1} W_{[m]}\right)
\Bigg/\left(C_{n}(W)\cap  \left(\coprod_{k<\Re(m)-N_{0}\le k+1} W_{[m]}\right)\right)$$
as a subspace of the finite-dimensional space $W/C_{n}(W)$ is also finite-dimensional. 
Thus by \eqref{C-2-gr-wk-1} and  the finite-dimensionality of this quotient, we obtain 
\begin{align*}
&\dim \left(\coprod_{k<\Re(m)-N_{0}\le k+1} W_{[m]}\right)\nn
&\quad \le \dim \left(C_{n}(W)\cap  \left(\coprod_{k<\Re(m)-N_{0}\le k+1} W_{[m]}\right)\right)\nn
&\quad\quad 
+ \dim \left(\coprod_{k<\Re(m)-N_{0}\le k+1} W_{[m]}\right)
\Bigg/\left(C_{n}(W)\cap  \left(\coprod_{k<\Re(m)-N_{0}\le k+1} W_{[m]}\right)\right)\nn
&\quad<\infty,
\end{align*}
proving \eqref{C-2-gr-wk-0.5}. 

We prove Property 2 now.   Take a finite-dimensional 
subspace  $M$ of $W$ such that $W=C_{n}(W)+ M$. 
But every finite-dimensional 
subspace of $W$ must be in $\coprod_{\Re(m)\le N}W_{[m]}$
for some $N\in \N$. In particular,  there exists $N_{M}\in \N$ such that $M\subset 
\coprod_{\Re(m)\le N_{M}}W_{[m]}$ and thus 
$\coprod_{\Re(m)> N_{M}}W_{[m]}
\subset C_{n}(W)$. Since 
$\coprod_{\Re(m)\le N_{M}}W_{[m]}$ is also 
finite-dimensional, we take 
$M=\coprod_{\Re(m)\le N_{M}}W_{[m]}$ from now on. 

Denote the space spanned by
elements of the form \eqref{C-2-gr-wk-0.1} by $\widetilde{W}$. What we want to prove is 
$W=\widetilde{W}$. 
We need only prove that $W_{[m]}\subset \widetilde{W}$ 
for every $m\in\C$. For each $m\in \C$ such that $W_{[m]}\ne 0$, there exists a unique $N_{m}\in \Z$ such that 
$N_{m}-1<\Re(m)\le N_{m}$. We use induction on $N_{m}$.  
When $N_{m}\le N_{M}$,  $W_{[m]}\subset 
M\subset \widetilde{W}$.
Assume that 
for $N_{m}\le p\in N_{M}+\N$, $W_{[m]}\subset \widetilde{W}$.
Then in the case $N_{m}=p+1$, since $\Re(m)>N_{m}-1=p\ge N_{M}$ and 
$\coprod_{\Re(m)> N_{M}}W_{[m]}
\subset C_{n}(W)$, we obtain
$W_{[m]}
\subset C_{n}(W)$. Then $W_{[m]}$  is spanned by elements of the form $v_{\alpha-n}w$ for
$v\in V_{(m+\alpha-n+1-l)}^{[\alpha]}$ and $w\in W_{[l]}$
for $l\in \C$ satisfying $\Re(l)\le p$ and $m+\alpha-n+1-l\in \Z_{+}$.  Since $\Re(l)\le p$, we have 
$N_{p}\le p$. By induction assumption, $w\in \widetilde{W}$. So $w$ is a linear combination of elements 
of the form $v^{(1)}_{\alpha_{1}-n}
\cdots v^{(i)}_{\alpha_{i}-n}\tilde{w}$
for homogeneous $v^{(1)}\in V^{[\alpha_{1}]}_{+}, \dots, v^{(i)}\in V^{[\alpha_{i}]}_{+}$ 
and $\tilde{w}\in M$.
Then $W_{\l m\r}$ is spanned by elements of the form $v_{\beta-n}v^{(1)}_{\alpha_{1}-n}
\cdots v^{(i)}_{\alpha_{i}-n}\tilde{w}$ for $v\in V_{(l)}^{[\beta]}$, 
$v^{(1)}\in V^{[\alpha_{1}]}_{+}, \dots, v^{(i)}\in V^{[\alpha_{i}]}_{+}$, and  $\tilde{w}\in M$.
So $W_{[m]}\subset \widetilde{W}$, proving Property 2. 

Finally, we prove Property 3. 
Since $W$ is quasi-finite-dimensional, $W_{[m]}$ for $m\in \C$ is Since $W$ is in particular
finite-dimensional. If
$W$ has an action of another automorphism $h$ of $V$ such that the $h$-action commutes
with $L_{W}(0)$, $W_{[m]}$ is invariant under the action of $h$. So $W_{[m]}$ can be decomposed as 
a direct sum of generalized eigenspaces of the action of $h$, and thus $W$ can also be decomposed as a direct sum of 
generalized eigenspaces of the action of $h$. 
\end{proof}

\begin{rema}
{\rm From Property 1 in Proposition \ref{C-2-gr-wk}, we see that
$C_{n}$-cofinite lower-bounded generalized $g$-twisted $V$-modules, 
$C_{n}$-cofinite grading-restricted generalized $g$-twisted $V$-modules, and 
$C_{n}$-cofinite quasi-finite-dimensional generalized $g$-twisted $V$-modules are
the same. In this paper, we will call these $g$-twisted modules 
$C_{n}$-cofinite grading-restricted generalized $g$-twisted $V$-modules, but note that 
these $g$-twisted modules  are  quasi-finite-dimensional.}
\end{rema}

For the next result and the twisted Nahm inequality to be proved in the next section, 
we  need the notion of surjectivity of a twisted intertwining operator. 
Let $W_{1}$, $W_{2}$,  $W_{3}$ be weak $g_{1}$-, $g_{2}$-, $g_{3}$-twisted $V$-modules, respectively. 
A twisted intertwining operator $\Y$ of type $\binom{W_{3}}{W_{1}W_{2}}$
is said to be surjective if $W_{3}$ is spanned by 
the coefficients of $\Y(w_{1}, x)w_{2}$ for $w_{1}\in W_{1}$ and $w_{2}\in W_{2}$.
If for a weak $V$-module $W_{3}$, there is an surjective twisted intertwining 
operator $\Y$ of type $\binom{W_{3}}{W_{1}W_{2}}$, we say that 
$(W_{3}, \Y)$ is a {\it weak surjective product of $W_{1}$ and $W_{2}$}. For simplicity, we shall 
often call $W_{3}$ a weak surjective product of $W_{1}$ and $W_{2}$. 
If $W_{3}$ is generalized $g_{3}$-twisted $V$-module or grading-restricted 
generalized $g_{3}$-twisted $V$-module or other classes of $g_{3}$-twisted $V$-modules, we also use the terms
{\it generalized surjective product of  $W_{1}$ and $W_{2}$}, 
{\it grading-restricted generalized surjective product  
 of $W_{1}$ and $W_{2}$} and so on.

\begin{prop}\label{int-weak-generalized}
Let $W_{1}$ and $W_{2}$ be generalized $g_{1}$ and $g_{2}$-twisted $V$-modules, respectively,
and $W_{3}$ a weak $g_{3}$-twisted $V$-module. If $W_{3}$ is a weak surjective product of 
$W_{1}$ and $W_{2}$, then $W_{3}$ is also a generalized $g_{3}$-twisted $V$-module. 
\end{prop}
\begin{proof}
Let $V^{g_{1}, g_{2}, g_{3}}$ be the fixed point subalgebra of $V$ under the group 
generated by $g_{1}$, $g_{2}$ and $g_{3}$. Then 
$W_{1}$ and $W_{2}$ are also generalized $V^{g_{1}, g_{2}, g_{3}}$-modules, respectively,
and $W_{3}$  also a weak $V^{g_{1}, g_{2}, g_{3}}$-module. Let $\Y$ be the  
 surjective twisted intertwining operator $\Y$ of 
type $\binom{W_{3}}{W_{1}W_{2}}$ for the weak surjective product  $W_{3}$ of 
$W_{1}$ and $W_{2}$. Then $\Y$ is  an intertwining operator 
of the same type when $W_{1}$, $W_{2}$, 
and $W_{3}$ are viewed as generalized or weak $V^{g_{1}, g_{2}, g_{3}}$-modules. 
Moreover, since $\Y$ as a twisted intertwining operator among the generalized or weak 
twisted $V$-modules is surjective, it is also surjective as an intertwining operator among 
the generalized or weak $V^{g_{1}, g_{2}, g_{3}}$-modules. 
Then by Proposition 3.4 in \cite{H-C1-vtc}, $W_{3}$ is a generalized $V^{g_{1}, g_{2}, g_{3}}$-module. 
Thus $W_{3}$ is a generalized $g_{3}$-twisted $V$-module.
\end{proof}

\setcounter{equation}{0}

\section{A twisted Nahm inequality}

In this section, by using the same method as in \cite{H-C1-vtc}, 
we prove the following twisted Nahm inequality in the case that the automorphisms of $V$ involved 
commute, but in general can be of infinite orders and do not have to act semisimply on $V$:

\begin{thm}\label{n-W-1-W-2}
Let $g_{1}$ and $g_{2}$ be commuting automorphisms 
of $V$ and $W_{1}$, $W_{2}$ lower-bounded generalized 
$g_{1}$, $g_{2}$-twisted $V$-modules, respectively.  Then 
 for a generalized surjective product 
$W_{3}$ of $W_{1}$ and $W_{2}$ and $p, q\in 2+\N$, 
\begin{equation}\label{nahm-ineq}
\dim (W_{3}/
C_{\min(p, q)}(W_{3}))\le \dim (W_{1}/
C_{p}(W_{1}))\dim (W_{2}/
C_{q}(W_{2})).
\end{equation}
\end{thm}
\begin{proof}
In the case that at least one of $\dim (W_{1}/
C_{p}(W_{1}))$ and $\dim (W_{2}/
C_{q}(W_{2}))$ is $\infty$, \eqref{nahm-ineq} holds. 
So we need only prove \eqref{nahm-ineq} in the case that 
$\dim W_{1}/
C_{p}(W_{1}), \dim W_{2}/
C_{q}(W_{2})<\infty$, that is, in the case that 
$W_{1}$ and $W_{2}$ are $C_{p}$-cofinite and $C_{q}$-cofinite
lower-bounded generalized twisted $V$-modules, respectively. 
In this case,  the properties in Proposition \ref{C-2-gr-wk} hold for $W_{1}$ and $W_{2}$. 
In particular, $W_{1}$ and $W_{2}$  are quasi-finite-dimensional generalized twisted $V$-modules. 

Let $M_{1}$ and $M_{2}$ be finite-dimensional graded subspaces of $W_{1}$ and $W_{2}$,
respectively,  such that 
$W_{1}=C_{p}(W_{1})\oplus M_{1}$ and $W_{2}=C_{q}(W_{2})\oplus M_{2}$.
Then $\dim (W_{1}/C_{p}(W_{1}))=\dim M_{1}$ and $\dim (W_{2}/C_{q}(W_{2}))=\dim M_{2}$.

Since $W_{3}$
is a generalized $g_{1}g_{2}$-twisted $V$-module, $W_{3}$ is $\C$-graded and 
we have the contragredient $W_{3}'$ of 
$W_{3}$. 
Let $M_{3}$ be a graded subspace of $W_{3}$ 
such that $W_{3}=C_{\min(p, q)}(W_{3})\oplus M_{3}$. Then 
$W_{3}/C_{\min(p, q)}(W_{3})$ 
is isomorphic to $M_{3}$. 
We need only prove 
$\dim M_{3}\le \dim M_{1} \dim M_{2}$. 

Let 
$$C_{\min(p, q)}(W_{3})^{\bot}=\{w_{3}'\in W_{3}'\mid \langle w_{3}', 
C_{\min(p, q)}(W_{3})\rangle=0\}
\subset W_{3}'.$$
and let $M_{3}'$ be the graded dual of $M_{3}$ with respect to the 
$\C$-grading induced from the one on $W_{3}$. 
We define a linear map $r: C_{\min(p, q)}(W_{3})^{\bot}\to M_{3}'$ to 
be the restriction map sending elements of $C_{\min(p, q)}(W_{3})^{\bot}\subset W_{3}'$ 
to  their restrictions to $M_{3}$, that is, 
$(r(w_{3}'))(w_{3})=\langle w_{3}',  w_{3}\rangle$
for $w_{3}'\in C_{\min(p, q)}(W_{3})^{\bot}$ and $w_{3}\in M_{3}$. 
If $r(w_{3}')=0$ for $w_{3}'\in C_{\min(p, q)}(W_{3})^{\bot}$, 
then $\langle w_{3}', w_{3}\rangle=0$ for all $w_{3}\in M_{3}$. 
But by definition, $\langle w_{3}', C_{\min(p, q)}(W_{3})\rangle =0$. Since 
$W_{3}=C_{\min(p, q)}(W_{3})\oplus M_{3}$, we obtain
$\langle w_{3}', w_{3}\rangle=0$ for all $w_{3}\in W_{3}$. 
Thus $w_{3}'=0$, proving that $r$ is injective. 
Given $w_{3}'\in M_{3}'$, we extend it to an element
$\bar{w}_{3}'\in W_{3}'$ by 
$\langle \bar{w}_{3}', \tilde{w}_{3}+w_{3}\rangle=
\langle w_{3}', w_{3}\rangle$ for $\tilde{w}_{3}\in 
C_{\min(p, q)}(W_{3})$ and $w_{3}\in M_{3}$. By definition, 
$\bar{w}_{3}'\in C_{\min(p, q)}(W_{3})^{\bot}$ and $r(\bar{w}_{3}')=w_{3}'$, 
proving $r$ is surjective. 
We have proved that 
$r$ is a linear isomorphism. Therefore we need only prove that
$\dim C_{\min(p, q)}(W_{3})^{\bot}\le \dim M_{1} \dim M_{2}$.

Let $\Y$ be a surjective twisted intertwining operator of type $\binom{W_{3}}{W_{1}W_{2}}$.
Then we have
$$\Y(w_{1}, x)w_{2}=\sum_{k=0}^{K}\sum_{n\in \C} \Y_{n, k}
(w_{1})w_{2} x^{-n-1}(\log x)^{k}$$
for $w_{1}\in W_{1}$ and $w_{2}\in W_{2}$. 
For homogeneous $w_{1}\in W_{1}$ and $w_{2}\in W_{2}$, we have
$$\wt \Y_{n, k}(w_{1})w_{2}=\wt w_{1}-n-1+\wt w_{2}.$$ 
Hence for homogeneous $w_{1}\in W_{1}$, $w_{2}\in W_{2}$, and  $w_{3}'\in W_{3}'$,
$$\langle w_{3}',  \Y(w_{1}, x)w_{2}\rangle
=\sum_{k=0}^{K}\langle w_{3}',  \Y_{\swt w_{1}+\swt w_{2}-\swt w_{3}'-1, k}
(w_{1})w_{2}\rangle
x^{-\swt w_{1}-\swt w_{2}+\swt w_{3}'}(\log x)^{k}.$$

Fix $z\in \C^{\times}$. Recall from Section 2 that 
we use $\log z$  to denote 
$\log |z|+i\arg z$, where $0\le \arg z<2\pi$ and the $0$-th value $\Y(w_{1}, z)w_{2}= \Y^{0}(w_{1}, z)w_{2}$
of $ \Y(w_{1}, x)w_{2}$ at $x=z$. Then
$$\langle w_{3}' , \Y(w_{1}, z)w_{2}\rangle
=\sum_{k=0}^{K}\langle w_{3}',  \Y_{\swt w_{1}+\swt w_{2}-\swt w_{3}'-1, k}
(w_{1})w_{2}\rangle
e^{(-\swt w_{1}-\swt w_{2}+\swt w_{3}')\log z}(\log z)^{k}$$
is well defined. 
We define a linear map 
$f: C_{\min(p, q)}(W_{3})^{\bot}\to (M_{1}\otimes M_{2})^{*}$ by 
$$(f(w_{3}'))(w_{1}\otimes w_{2})
=\langle w_{3}', \Y(w_{1}, z)w_{2}\rangle$$
for $w_{1}\in M_{1}$, $w_{2}\in M_{2}$ and $w_{3}'\in C_{\min(p, q)}(W_{3})^{\bot}$. 
To prove 
$$\dim C_{\min(p, q)}(W_{3})^{\bot}\le \dim M_{1} \dim M_{2}=\dim (M_{1}\otimes M_{2})^{*},$$
we need only prove that $f$ is injective. 

Assume that $f(w_{3}')=0$ for an element $w_{3}'\in C_{\min(p, q)}(W_{3})^{\bot}$.
Then by the definition of $f$, for $w_{1}\in M_{1}$, $w_{2}\in M_{2}$,
\begin{equation}\label{int-w-1-w-2}
\langle w_{3}', \Y(w_{1}, z)w_{2}\rangle=0
\end{equation}
We now prove \eqref{int-w-1-w-2} for all $w_{1}\in W_{1}$, $w_{2}\in W_{2}$. 

Since $W_{1}$ and $W_{2}$ are 
quasi-finite-dimensional generalized $V$-modules, we have $W_{1}=\coprod_{m\in \C}(W_{1})_{[m]}$ and 
$W_{2}=\coprod_{n\in \C}(W_{2})_{[n]}$ such that for $N\in \N$, 
$$\dim \left(\coprod_{\Re(m)\le N}(W_{1})_{[m]}\right),
\dim \left(\coprod_{\Re(n)\le N}(W_{2})_{[n]}\right)<\infty.$$ 
In fact,  there exist $N_{1}^{0}\in \Z$ such that 
$(W_{1})_{[m]}=0$ for $\Re(m)<N_{1}^{0}$ and there exists $m\in \C$ satisfying $N_{1}^{0}\le \Re(m)<N_{1}^{0}+1$
such that $(W_{1})_{[m]}\ne 0$. Similarly we have such an $N_{2}^{0}$ for $W_{2}$. We want to prove  \eqref{int-w-1-w-2}
for all $w_{1}\in \coprod_{\Re(m)\le N_{1}^{0}+N_{1}}(W_{1})_{[m]}$ and $w_{2}\in 
\coprod_{\Re(n)\le N_{2}^{0}+N_{2}}(W_{1})_{[n]}$
for $N_{1}, N_{2}\in \N$. 
We use induction on $N_{1}+N_{2}$ to 
prove \eqref{int-w-1-w-2}.

When $N_{1}+N_{2}=0$, $N_{1}=N_{2}=0$. 
Since $W_{1}=C_{p}(W_{1})\oplus M_{1}$ and $W_{2}=C_{q}(W_{2})\oplus M_{2}$,
there exist homogeneous $u^{(k)}\in V_{g_{1}}^{[\alpha_{k}]}\cap V_{+}, v^{(l)}\in V_{g_{2}}^{[\beta_{l}]}\cap V_{+}$,
$\tilde{w}_{1}^{(k)} \in (W_{1})_{[p_{k}]}$, $\tilde{w}_{2}^{(l)} \in (W_{2})_{[q_{l}]}$, 
for $k=1, \dots, s$, $l=1, \dots, t$,
$\tilde{w}_{1}\in M_{1}\cap (W_{1})_{[ N_{1}^{0}]}$, and
$\tilde{w}_{2}\in M_{2}\cap (W_{2})_{[N_{2}^{0}]}$ such that 
\begin{align}
w_{1}&=\sum_{k=1}^{s}u^{(k)}_{\alpha_{k}-p}\tilde{w}^{(k)}_{1}
+\tilde{w}_{1},\label{nahm-ineq-1}\\
w_{2}&=\sum_{l=1}^{t}v^{(l)}_{\beta_{l}-q}\tilde{w}^{(l)}_{2}+\tilde{w}_{2}.
\label{nahm-ineq-2}
\end{align}
We must have $u^{(k)}_{\alpha_{k}-p}\tilde{w}^{(k)}_{1}\in (W_{1})_{[N_{1}^{0}]}$,
since $w_{1}\in (W_{1})_{[N_{1}^{0}]}$. 
On the other hand, we have 
$u^{(k)}_{\alpha_{k}-p}\tilde{w}^{(k)}_{1}\in (W_{1})_{[ \swt u^{(k)}-\alpha_{k}+
p-1+p_{k}]}$. So we obtain
$\wt u^{(k)}-\alpha_{k}+
p-1+p_{k}=N_{1}^{0}$, or equivalently, 
$p_{k}=N_{1}^{0}-\wt u^{(k)}+\alpha_{k}-p+1$. Since $\wt u^{(k)}\in \Z_{+}$ and 
$\Re(\alpha_{k}\in [0, 1)$, we have $\Re(p_{k})=N_{1}^{0}-\wt u^{(k)}+\Re(\alpha_{k})-p+1<N_{1}^{0}$.
Hence $(W_{1})_{[p_{k}]}=0$ and therefore
$u^{(k)}_{\alpha_{k}-p}\tilde{w}^{(k)}_{1}= 0$ for 
$k=1, \dots, s$. Similarly, we can prove $v^{(l)}_{\beta_{l}-q}\tilde{w}^{(l)}_{2}=0$.
Thus $w_{1}=\tilde{w}_{1}\in M_{1}$ and $w_{2}=\tilde{w}_{2}\in M_{2}$. In this case, \eqref{int-w-1-w-2} is true. 

Assume that when $N_{1}+N_{2}<m \in \N$, \eqref{int-w-1-w-2} is true. 
In the case $N_{1}+N_{2}=m$,  we consider $w_{1}\in (W_{1})_{[n_{1}]}$
and $w_{2}\in (W_{2})_{[n_{2}]}$ for $n_{1}, n_{2}\in \C$ such that $\Re(n_{1})\le N_{1}^{0}+N_{1}$, 
$\Re(n_{2})\le N_{2}^{0}+N_{2}$, and $N_{0}^{1}+N_{2}^{0}+m-1<\Re(n_{1}+n_{2})\le N_{0}^{1}+n_{2}^{0}+ N_{1}+N_{2}$. 
We still have \eqref{nahm-ineq-1} and \eqref{nahm-ineq-2} for
homogeneous $u^{(k)}\in V_{g_{1}}^{[\alpha_{k}]}\cap V_{+}, v^{(l)}\in V_{g_{2}}^{[\beta_{l}]}\cap V_{+}$,
$\tilde{w}_{1}^{(k)} \in (W_{1})_{[p_{k}]}$, $\tilde{w}_{2}^{(l)} \in (W_{2})_{[ q_{l}]}$, 
for $k=1, \dots, s$, $l=1, \dots, t$, $\tilde{w}_{1}\in M_{1}\cap (W_{1})_{[n_{1}]}$, and
$\tilde{w}_{2}\in M_{2}\cap (W_{2})_{[n_{2}]}$.
In fact, we can always find $u^{(k)}\in V_{g_{1}, g_{2}}^{[\alpha_{k}, \hat{\beta}_{k}]}\cap V_{+}, 
v^{(l)}\in V_{g_{1}, g_{2}}^{[\hat{\alpha}_{l}, \beta_{l}]}\cap V_{+}$ and the corresponding 
$\tilde{w}_{1}^{(k)}$, $\tilde{w}_{2}^{(l)}$, $\tilde{w}_{1}$, and $\tilde{w}_{2}$ such that 
 \eqref{nahm-ineq-1} and \eqref{nahm-ineq-2} hold. 
 In this case, similarly to the proof above in the case $N_{1}=N_{2}=0$, 
we have $\wt u^{(k)}-\alpha_{k}+p-1+p_{k}=n_{1}$, 
or equivalently, $\wt u^{(k)}+p_{k}=n_{1}+\alpha_{k}-p+1$.
Similarly, we also have $\wt v^{(l)}-\beta_{l}+q-1+q_{l}=n_{2}$
or equivalently, $\wt v^{(l)}+q_{l}=n_{2}+\beta_{l}-q+1$. 

We need to use \eqref{jacobi-twisted-int-2}. We first rewrite \eqref{jacobi-twisted-int-2} as
\begin{align}\label{jacobi-twisted-int-4}
&\sum_{k\in \N}\sum_{l\in \N}\binom{\alpha_{1}-n_{1}}{k}(-x)^{k+l}
\left(\binom{\mathcal{N}_{g_{1}}}{l}x^{-\mathcal{N}_{g_{2}}}v\right)
_{\alpha_{1}+\alpha_{2}-n_{1}-n_{2}-k-l}\Y(w_{1}, x)\nn
&\quad \quad-\sum_{k\in \N}\sum_{l\in \N}\binom{\alpha_{1}-n_{1}}{k}
e^{\pi \i (\alpha_{1}-n_{1}-k-l)}x^{\alpha_{1}-n_{1}-k-l}\cdot\nn
&\quad\quad\quad\quad\quad\quad\quad\cdot \Y(w_{1}, x)\left(e^{\pi \i\mathcal{N}_{g_{1}}}\binom{\mathcal{N}_{g_{1}}}{l}
x^{\mathcal{N}_{g_{1}}-\mathcal{N}_{g_{2}}}v\right)_{\alpha_{2}-n_{2}+k+l}\nn
&\quad=\sum_{k\in \N}\sum_{l\in \N}\binom{\alpha_{2}-n_{2}}{k}
x^{\alpha_{2}-n_{2}-k-l}
\Y\left(\left(\binom{\mathcal{N}_{g_{2}}}{l}v\right)_{\alpha_{1}-n_{1}+k+l}w_{1}, x\right)
\end{align}
by replacing $v$ by $x^{-\mathcal{N}_{g_{2}}}v$ in \eqref{jacobi-twisted-int-2}.
Then using \eqref{jacobi-twisted-int-4} with $v$, $\alpha_{1}$, $\alpha_{2}$, $n_{1}$, $n_{2}$, $w_{1}$, and $w_{2}$ 
in \eqref{jacobi-twisted-int-4} taken to be 
$u^{(k)}$, $\alpha_{k}$, $\hat{\beta}_{k}$, $p$, $1$,
$\tilde{w}^{(k)}_{1}$,  respectively, and changing the summation indices from $k, l$ to $i, j$, 
we obtain
\begin{align}\label{nahm-ineq-3}
&\langle w_{3}',  \Y(u^{(k)}_{\alpha_{k}-p}\tilde{w}^{(k)}_{1}, x)w_{2}\rangle\nn
&\quad =\sum_{i\in \N}\sum_{j\in \N}\binom{\alpha_{k}-p}{i}(-1)^{i+j}
x^{i+j-\hat{\beta}_{k}+1}\cdot\nn
&\quad\quad\quad\quad\quad\quad\cdot 
\left\langle w_{3}', \left(\binom{\mathcal{N}_{g_{1}}}{j}x^{-\mathcal{N}_{g_{2}}}u^{(k)}\right)
_{\alpha_{k}+\hat{\beta}_{k}-p-1 -i-j}\Y(\tilde{w}^{(k)}_{1}, x)w_{2}\right\rangle\nn
&\quad\quad -\sum_{i\in \N}\sum_{j\in \N, j+i\ne 0}\binom{\hat{\beta}_{k}-1}{i}x^{-i-j}
\left\langle w_{3}',  \Y\left(\left(\binom{\mathcal{N}_{g_{2}}}{j}u^{(k)}\right)_{\alpha_{k}-p+i+j}
\tilde{w}^{(k)}_{1}, x\right)w_{2}\right\rangle\nn
&\quad\quad -\sum_{i\in \N}\binom{\alpha_{k}-p}{i}e^{\pi \i(\alpha_{k}
-p-i-j)} x^{\alpha_{k}-\hat{\beta}_{k}-p+1-i-j}\cdot\nn
&\quad\quad\quad\quad\quad\quad\cdot 
\langle w_{3}',  \Y(\tilde{w}^{(k)}_{1}, x)\left(e^{\pi \i\mathcal{N}_{g_{1}}}\binom{\mathcal{N}_{g_{1}}}{j}
x^{\mathcal{N}_{g_{1}}-\mathcal{N}_{g_{2}}}u^{(k)}\right)_{\hat{\beta}_{k}-1+i+j}w_{2}\rangle
\end{align}
We now prove that each term in the right-hand side of \eqref{nahm-ineq-3} is $0$. 

Since  $\alpha_{k}+\hat{\beta}_{k}=\sigma(\alpha_{k}, \hat{\beta}_{k})+\epsilon(\alpha_{k}, \hat{\beta}_{k})$,
$\alpha_{k}+\hat{\beta}_{k}-p-1 -i-j=\sigma(\alpha_{k}, 
\hat{\beta}_{k})+\epsilon(\alpha_{k}, \hat{\beta}_{k})-p-1 -i-j$, where by definition, $\sigma(\alpha_{k}, 
\hat{\beta}_{k})\in P_{V}^{g_{1}g_{2}}$ and $\epsilon(\alpha_{k}, \hat{\beta}_{k})-p-1 -i-j\le -p$. 
So 
$$ \left(\binom{\mathcal{N}_{g_{1}}}{j}x^{-\mathcal{N}_{g_{2}}}u^{(k)}\right)
_{\alpha_{k}+\hat{\beta}_{k}-p-1 -i-j}\Y(\tilde{w}^{(k)}_{1}, x)w_{2}\in C_{p}(V)\{x\}[\log x]
\subset C_{\min(p, q)}(V)\{x\}[\log x].$$
Since  $w_{3}'\in C_{\min(p, q)}(W_{3})^{\bot}$, we see that the first term in the right-hand side of 
\eqref{nahm-ineq-3} is $0$. 

Since $\binom{\mathcal{N}_{g_{2}}}{j}u^{(k)}$ is also homogeneous with respect to the 
weight grading and its weight is equal to $\wt u^{(k)}$, 
we have 
$$\left(\binom{\mathcal{N}_{g_{2}}}{j}u^{(k)}\right)_{\alpha_{k}-p+i+j}\tilde{w}^{(k)}_{1}\in (W_{1})
_{[\swt u^{(k)}-\alpha_{k}+p-1-i+j+p_{k}]}$$
for $i, j\in \N$, $i+j\ne 0$.  We have also proved above $\wt u^{(k)}-\alpha_{k}+p-1+p_{k}=n_{1}$. So we have 
$$\Re(\wt u^{(k)}-\alpha_{k}+p-1-i+p_{k}+n_{2})=\Re(n_{1}+n_{2})-i<\Re(n_{1}+n_{2})\le N_{0}^{1}+N_{2}^{0}+ m$$
for $i\in \Z_{+}$.
By the induction assumption, 
$$\left\langle w_{3}',  \Y\left(\left(\binom{\mathcal{N}_{g_{2}}}{j}u^{(k)}\right)_{\alpha_{k}-p+i+j}
\tilde{w}^{(k)}_{1}, x\right)w_{2}\right\rangle=0.$$
Hence the second term in the right-hand side of 
\eqref{nahm-ineq-3} is $0$. 

Since $w_{2}\in (W_{2})_{[n_{2}]}$ and the coefficients of $e^{\pi \i\mathcal{N}_{g_{1}}}\binom{\mathcal{N}_{g_{1}}}{j}
x^{\mathcal{N}_{g_{1}}-\mathcal{N}_{g_{2}}}u^{(k)}$ are homogeneous with weight $\wt u^{(k)}$, we have
$$\left(e^{\pi \i\mathcal{N}_{g_{1}}}\binom{\mathcal{N}_{g_{1}}}{j}
x^{\mathcal{N}_{g_{1}}-\mathcal{N}_{g_{2}}}u^{(k)}\right)
_{\hat{\beta}_{k}-1+i+j}w_{2}\in (W_{2})_{[ \swt u^{(k)}-\hat{\beta}_{k}-i-j+n_{2}]}.$$
We also know that $\tilde{w}_{1}^{(k)} \in (W_{1})_{[p_{k}]}$. Using $\wt u^{(k)}+p_{k}=n_{1}+\alpha_{k}-p+1$
proved above, $\Re(\alpha_{k}), \Re(\hat{\beta}_{k})\in [0, 1)$, and $p\in 2+\N$, we have 
\begin{align*}
\Re(p_{k}+ \wt u^{(k)}-\hat{\beta}_{k}-i+n_{2})&=\Re(n_{1}+n_{2}+\alpha_{k}-p+1-\hat{\beta}_{k}-i)\nn
&<\Re(n_{1}+n_{2})\nn
&\le N_{0}^{1}+N_{2}^{0}+ m.
\end{align*}
Then by the induction assumption, 
$$\langle w_{3}',  \Y(\tilde{w}^{(k)}_{1}, x)\left(e^{\pi \i\mathcal{N}_{g_{1}}}\binom{\mathcal{N}_{g_{1}}}{j}
x^{\mathcal{N}_{g_{1}}-\mathcal{N}_{g_{2}}}u^{(k)}\right)_{\hat{\beta}_{k}-1+i+j}w_{2}\rangle=0$$
 and thus the 
third term in the right-hand side of 
\eqref{nahm-ineq-3} is  $0$. 

We have proved that each term in the right-hand side of 
\eqref{nahm-ineq-3} is $0$. So the left-hand side of 
\eqref{nahm-ineq-3} is also $0$. 

Similarly, we first rewrite \eqref{jacobi-twisted-int-2} as
\begin{align}\label{jacobi-twisted-int-5}
&\sum_{k\in \N}\sum_{l\in \N}\binom{\alpha_{1}-n_{1}}{k}(-x)^{k+l}
\left(\binom{\mathcal{N}_{g_{1}}}{j}e^{\pi \i\mathcal{N}_{g_{1}}}x^{-\mathcal{N}_{g_{1}}}v\right)
_{\alpha_{1}+\alpha_{2}-n_{1}-n_{2}-k-l}\Y(w_{1}, x)\nn
&\quad \quad-\sum_{k\in \N}\sum_{l\in \N}\binom{\alpha_{1}-n_{1}}{k}
e^{\pi \i (\alpha_{1}-n_{1}-k-l)}x^{\alpha_{1}-n_{1}-k-l}
\Y(w_{1}, x)\left(\binom{\mathcal{N}_{g_{1}}}{j}v\right)_{\alpha_{2}-n_{2}+k+l}\nn
&\quad=\sum_{k\in \N}\sum_{l\in \N}\binom{\alpha_{2}-n_{2}}{k}
x^{\alpha_{2}-n_{2}-k-l}\Y\left(\left(\binom{\mathcal{N}_{g_{2}}}{j}x^{\mathcal{N}_{g_{2}}}
e^{\pi \i\mathcal{N}_{g_{1}}}x^{-\mathcal{N}_{g_{1}}}v\right)_{\alpha_{1}-n_{1}+k+l}w_{1}, x\right)
\end{align}
by replacing $v$ in \eqref{jacobi-twisted-int-2} by $e^{\pi \i\mathcal{N}_{g_{1}}}x^{-\mathcal{N}_{g_{1}}}v$.
Then using \eqref{jacobi-twisted-int-5} with $v$, $\alpha_{1}$, $\alpha_{2}$, 
$n_{1}$, $n_{2}$, and $w_{1}$ taken to be 
$v^{(l)}$, $\hat{\alpha}_{l}$, $\beta_{l}$, $1$, $q$,
$w_{1}$, respectively, 
we have
\begin{align}\label{nahm-ineq-4}
&\langle w_{3}',  \Y(w_{1}, x)v^{(l)}_{\beta_{l}-q}\tilde{w}^{(l)}_{2}\rangle\nn
&\quad =\sum_{i\in \N}\sum_{j\in \N}\binom{\hat{\alpha}_{l}-1}{i}(-1)^{i+j}e^{\pi \i(\hat{\alpha}_{l}
-1)}
x^{i+j-\hat{\alpha}_{l}+1}\cdot\nn
&\quad\quad\quad\quad\quad\quad\cdot 
\left\langle w_{3}', \left(\binom{\mathcal{N}_{g_{1}}}{j}e^{\pi \i\mathcal{N}_{g_{1}}}
x^{-\mathcal{N}_{g_{1}}}v^{(l)}\right)_{\hat{\alpha}_{l}+\beta_{l}-1-q -i-j}\Y(w_{1}, x)\tilde{w}^{(l)}_{2}\right\rangle\nn
&\quad\quad -\sum_{i\in \N}\sum_{j\in \N}\binom{\beta_{l}-q}{i}e^{\pi \i(\hat{\alpha}_{l}
-1)}x^{-\hat{\alpha}_{l}+\beta_{l}+1-q-i-j}\cdot\nn
&\quad\quad\quad\quad\quad\quad\cdot 
\left\langle w_{3}',  \Y\left(\left(\binom{\mathcal{N}_{g_{2}}}{j}x^{\mathcal{N}_{g_{2}}}
e^{\pi \i\mathcal{N}_{g_{1}}}x^{-\mathcal{N}_{g_{1}}}v^{(l)}\right)
_{\hat{\alpha}_{l}-1+i+j}w_{1}, x\right)\tilde{w}^{(l)}_{2}\right\rangle\nn
&\quad\quad -\sum_{i\in \N}\sum_{j\in \N, j+i\ne 0}\binom{\hat{\alpha}_{l}-1}{i}e^{\pi \i (i+j)} x^{-i-j}
\left\langle w_{3}',  \Y(w_{1}, x)\left(\binom{\mathcal{N}_{g_{1}}}{j}v^{(l)}\right)_{\beta_{l}-q+i+j}\tilde{w}^{(l)}_{2}
\right\rangle
\end{align}
We now prove that the right-hand side of \eqref{nahm-ineq-4} is $0$.

Since $\hat{\alpha}_{l}+\beta_{l}=\sigma(\hat{\alpha}_{l}, \beta_{l})+\epsilon(\hat{\alpha}_{l}, \beta_{l})$,
$\hat{\alpha}_{l}+\beta_{l}-1-q -i=\sigma(\hat{\alpha}_{l}, \beta_{l})+\epsilon(\hat{\alpha}_{l}, \beta_{l})
-1-q -i$, where by definition, $\sigma(\hat{\alpha}_{l}, \beta_{l})\in P_{V}^{g_{1}g_{2}}$ 
and $\epsilon(\hat{\alpha}_{l}, \beta_{l})-1-q -i\le -q$. 
So 
\begin{align*}
\left(\binom{\mathcal{N}_{g_{1}}}{j}e^{\pi \i\mathcal{N}_{g_{1}}}
x^{-\mathcal{N}_{g_{1}}}v^{(l)}\right)_{\hat{\alpha}_{l}+\beta_{l}-1-q -i-j}\Y(w_{1}, x)\tilde{w}^{(l)}_{2}
&\in C_{q}(V)\{x\}[\log x]\nn
&\subset C_{\min(p, q)}(V)\{x\}[\log x].
\end{align*}
Since  $w_{3}'\in C_{\min(p, q)}(W_{3})^{\bot}$, we see that the first term in the right-hand side of 
\eqref{nahm-ineq-4} is $0$. 

Since $w_{1}\in (W_{1})_{[n_{1}]}$ and the coefficients of $\binom{\mathcal{N}_{g_{2}}}{j}x^{\mathcal{N}_{g_{2}}}
e^{\pi \i\mathcal{N}_{g_{1}}}x^{-\mathcal{N}_{g_{1}}}v^{(l)}$ are homogeneous with weight $\wt v^{(l)}$, we have
$$\left(\binom{\mathcal{N}_{g_{2}}}{j}x^{\mathcal{N}_{g_{2}}}
e^{\pi \i\mathcal{N}_{g_{1}}}x^{-\mathcal{N}_{g_{1}}}v^{(l)}\right)_{\hat{\alpha}_{l}-1+i+j}w_{1}
\in (W_{1})_{[ \swt v^{(l)}-\hat{\alpha}_{l}-i-j+n_{1}]}.$$
We also know that $\tilde{w}_{2}^{(l)} \in (W_{2})_{[q_{l}]}$. Using $\wt v^{(l)}+q_{l}=n_{2}+\beta_{l}-q+1$
proved above, $\Re(\hat{\alpha}_{l}), \Re(\beta_{l})\in [0, 1)$, and $q\in 2+\N$, we have 
\begin{align*}
\Re(\wt v^{(l)}-\hat{\alpha}_{l}-i-j+n_{1}+q_{l})&=\Re(n_{1}+n_{2}-\hat{\alpha}_{l}+\beta_{l}-q+1-i-j)\nn
&<\Re(n_{1}+n_{2})\nn
&\le N_{0}^{1}+N_{2}^{0}+ m.
\end{align*}
Then by the induction assumption, 
$$\left\langle w_{3}',  \Y\left(\left(\binom{\mathcal{N}_{g_{2}}}{j}x^{\mathcal{N}_{g_{2}}}
e^{\pi \i\mathcal{N}_{g_{1}}}x^{-\mathcal{N}_{g_{1}}}v^{(l)}\right)
_{\hat{\alpha}_{l}-1+i+j}w_{1}, x\right)\tilde{w}^{(l)}_{2}\right\rangle=0$$
 and thus the second term in the right-hand side of 
\eqref{nahm-ineq-4} is $0$. 

Since $\binom{\mathcal{N}_{g_{1}}}{j}v^{(l)}$ is homogeneous with weight equal to $\wt v^{(l)}$, 
$$\left(\binom{\mathcal{N}_{g_{1}}}{j}v^{(l)}\right)_{\beta_{l}-q+i+j}\tilde{w}^{(l)}_{2}\in (W_{2})
_{[\swt v^{(l)}-\beta_{l}+q-1-i-j+q_{l}]}$$
for  $i, j\in \N$, $i+j\ne 0$.  We have proved above $\wt v^{(l)}-\beta_{l}+q-1+q_{l}=n_{2}$. So we have 
$$\Re(n_{1}+\wt v^{(l)}-\beta_{l}+q-1-i-j+q_{l})=\Re(n_{1}+n_{2})-i-j<\Re(n_{1}+n_{2})\le N_{0}^{1}+N_{2}^{0}+ m$$
for $i, j\in \N$, $i+j\ne 0$.
By the induction assumption, 
$$\left\langle w_{3}',  \Y(w_{1}, x)\left(\binom{\mathcal{N}_{g_{1}}}{j}v^{(l)}\right)_{\beta_{l}-q+i+j}\tilde{w}^{(l)}_{2}
\right\rangle=0.$$
Thus the third term in the right-hand side of 
\eqref{nahm-ineq-4} is $0$. 

We have proved that the right-hand side of \eqref{nahm-ineq-4} is  $0$. 
So we obtain the left-hand side of \eqref{nahm-ineq-4} is $0$. 

Now we have
\begin{align*}
\langle w_{3}', \Y(w_{1}, z)w_{2}\rangle
&=\sum_{k=1}^{s}\langle w_{3}',  \Y(u^{(k)}_{\alpha_{k}-2}\tilde{w}^{(k)}_{1}, x)w_{2}\rangle
+\langle w_{3}',  \Y(\tilde{w}_{1}, z)w_{2}\rangle\nn
&=\langle w_{3}',  \Y(\tilde{w}_{1}, z)w_{2}\rangle\nn
&=\sum_{l=1}^{t}\langle w_{3}',  \Y(\tilde{w}_{1}, z)v^{(l)}_{-1}\tilde{w}^{(l)}_{2}\rangle
+\langle w_{3}',  \Y(\tilde{w}_{1}, z)\tilde{w}_{2}\rangle\nn
&=\langle w_{3}',  \Y(\tilde{w}_{1}, z)\tilde{w}_{2}\rangle\nn
&=0, 
\end{align*}
proving  \eqref{int-w-1-w-2} 
 for all $w_{1}\in W_{1}$ and 
$w_{2}\in W_{2}$.

Using the $L(-1)$-derivative property for $\Y$, we have 
$$\langle w_{3}', \Y(w_{1}, \xi)w_{2}\rangle
=\langle w_{3}', \Y(e^{(\xi -z)L_{W_{1}}(-1)}w_{1}, z)w_{2}\rangle=0$$
on the region $|\xi-z|<|z|$. But $\langle w_{3}', \Y(w_{1}, \xi)w_{2}\rangle$
is an analytic function of $\xi$. This analytic function equal to $0$ 
on a region means that it is equal to $0$ on its domain. So 
we obtain $\langle w_{3}', \Y(w_{1}, \xi)w_{2}\rangle=0$ on the region $\xi\ne 0$. 
Thus
$\langle w_{3}', \Y_{n, k}(w_{1})w_{2}\rangle=0$ for 
$w_{1}\in W_{1}$, 
$w_{2}\in W_{2}$,
$n\in \C$, $k=1, \dots, K$. 
Since $\Y$ is surjective, 
we must have $w_{3}'=0$, proving the injectivity of $f$. 
\end{proof}

\begin{rema}
{\rm The inequality \eqref{nahm-ineq} is an analogue of the (untwisted) Nahm inequality 
(3.8) in Theorem 3.5 in \cite{H-C1-vtc}. It is an ananlogue but not a generalization because 
$p, q\ne 1$ in \eqref{nahm-ineq}. 
In the case that $g_{1}$ and $g_{2}$ are of finite orders and $p=q=2$, \eqref{nahm-ineq} has been proved 
in \cite{YZ}.}
\end{rema}

\begin{rema}
{\rm By Proposition \ref{int-weak-generalized}, ``a generalized surjective product 
$W_{3}$'' in Theorem \ref{n-W-1-W-2} can be replaced by ``a weak surjective product 
$W_{3}$.'' }
\end{rema}

\setcounter{equation}{0}

\section{$P(z)$-tensor products of $C_{n}$-cofinite grading-restricted generalized twisted $V$-modules}

In this section, we consider two categories of grading-restricted generalized twisted $V$-modules 
for a M\"{o}bius vertex algebra or a quasi-vertex operator algebra $V$ (see \cite{FHL} and \cite{HLZ1})
and construct $P(z)$-tensor product bifunctors for these categories. 

In this section, $V$ is a M\"{o}bius vertex algebra. Note that automorphisms of 
$V$ also commute with the operator $L_{V}(1)$. Let $G$ be an abelian group 
of automorphisms of $V$ and $n\in 2+\N$. Let $\mathcal{C}_{n}^{G}$ be the category of 
$C_{n}$-cofinite grading-restricted generalized 
$g$-twisted $V$-modules for $g\in G$ and let $\widetilde{\mathcal{C}}_{n}^{G}$ be the category of 
$C_{n}$-cofinite grading-restricted generalized 
$g$-twisted $V$-modules with $G$-actions for $g\in G$. 
In Subsection 5.1, for $z\in \C^{\times}$,  we construct 
a $P(z)$-tensor product bifunctor for the category  $\mathcal{C}_{n}^{G}$. Then in Subsection 5.2, we
construct a $P(z)$-tensor product bifunctor for the category $\widetilde{\mathcal{C}}_{n}^{G}$.

In the case that $V$ is $C_{2}$-cofinite and of positive energy, by the main result in 
\cite{T1}, every finitely-generated lower-bounded generalized twisted $V$-module is $C_{n}$-cofinite
for $n\in 2+\N$. 
On the other hand, by Theorem \ref{C-2-gr-wk}, for $n\in 2+\N$, 
every $C_{n}$-cofinite lower-bounded
$g$-twisted generalized $V$-module is finitely generated. So in this case, 
the category $\mathcal{C}_{n}^{G}$ (or $\widetilde{\mathcal{C}}_{n}^{G}$)
is the same as the category of finitely-generated lower-bounded generalized twisted $V$-modules
(or the category of finitely-generated lower-bounded generalized twisted $V$-modules with $G$-actions).
Note that in this case, $\mathcal{C}_{n}^{G}$ and $\widetilde{\mathcal{C}}_{n}^{G}$ are independent of $n$. 

\subsection{$P(z)$-tensor product bifunctor for $\mathcal{C}_{n}^{G}$}

In this subsection, we construct a $P(z)$-tensor product bifunctor for $\mathcal{C}_{n}^{G}$.
We first recall twisted $P(z)$-intertwining maps in \cite{DH}. 

\begin{defn}
{\rm Let $g_{1}, g_{2}$ be commuting automorphisms of $V$, 
$W_{1}$, $W_{2}$, $W_{3}$ generalized $g_{1}$-, $g_{2}$-,
$g_{1}g_{2}$-twisted $V$-modules, respectively,
and  $z\in \C^{\times}$. A {\it twisted $P(z)$-intertwining map of type 
$\binom{W_{3}}{W_{1}W_{2}}$} is a linear 
map $I: W_{1}\otimes W_{2}\to \overline{W}_{3}$ 
given by
$I(w_{1}\otimes w_{2})=\Y(w_{1}, z)w_{2}$ for
$w_{1}\in W_{1}$ and $w_{2}\in W_{2}$, where 
$\Y$ is a twisted intertwining operator  of type 
$\binom{W_{3}}{W_{1}W_{2}}$.}
\end{defn}

Next we recall the definition of $P(z)$-tensor product in \cite{DH}
when the category is $\mathcal{C}_{n}^{G}$. 

\begin{defn}
{\rm Let $g_{1}, g_{2}\in G$
and let $W_{1}$ and $W_{2}$ be  grading-restricted generalized 
$g_{1}$- and $g_{2}$-twisted $V$-modules,
respectively,  in the category $\mathcal{C}_{n}^{G}$. A {\it $P(z)$-product of $W_{1}$ 
and $W_{2}$ in $\mathcal{C}_{n}^{G}$} 
is a pair $(W_{3}, I)$ consisting of a $C_{n}$-cofinite  grading-restricted generalized 
$g_{1}g_{2}$-twisted $V$-module  $W_{3}$ and 
a twisted $P(z)$-intertwining map $I$ of type 
$\binom{W_{3}}{W_{1}W_{2}}$. A {\it $P(z)$-tensor product 
of $W_{1}$ and $W_{2}$ in $\mathcal{C}_{n}^{G}$} is a $P(z)$-product $(W_{1}\boxtimes_{P(z)}W_{2},
\boxtimes_{P(z)})$ satisfying the following universal property:
For any $P(z)$-product $(W_{3}, I)$ of $W_{1}$ and $W_{2}$,
there exists a unique module map $f: W_{1}\boxtimes_{P(z)}W_{2}\to W_{3}$
such that we have the commutative diagram 
$$\begin{tikzcd}
W_{1}\otimes W_{2}\arrow[swap]{d}{\boxtimes_{P(z)}}
\arrow{r}{I}&\overline{W}_{3}\\
\overline{W_{1}\boxtimes_{P(z)}W_{2}}\arrow[swap]{ru}{\bar{f}}&
\end{tikzcd}$$
where $\bar{f}$ is the natural extension of $f$ to 
$\overline{W_{1}\boxtimes_{P(z)}W_{2}}$.}
\end{defn}

In \cite{DH}, a construction of a $P(z)$-tensor product bifunctor  using twisted intertwining 
maps is given based on some assumptions (see Assumption 4.4 in \cite{DH}) as mentioned 
above in the introduction. But 
the first two assumptions in  Assumption 4.4 in \cite{DH} do not hold for the 
category $\mathcal{C}_{n}^{G}$. On the other hand, many of the results and their proofs in \cite{DH} 
still hold in the general setting in our case. 
We now give a construction of the $P(z)$-tensor product bifunctor for the 
category $\mathcal{C}_{n}^{G}$ by using the results in \cite{DH} that still hold, the results in Sections 3 and 4,
and the method in \cite{H-C1-vtc}. Note that for a lower-bounded generalized twisted $V$-module, we 
still have the contragredient of this module. 
 
We recall the construction in \cite{DH} in the case that the category is 
$\mathcal{C}_{n}^{G}$, even though  Assumption 4.4 in \cite{DH} does hold for the 
category $\mathcal{C}_{n}^{G}$. Given a $P(z)$-product $(W_{3}, I)$ of $W_{1}$ and $W_{2}$  in $\mathcal{C}_{n}^{G}$,
for $w_{3}'\in W_{3}'$, we have an element $\lambda_{I, w_{3}'}\in 
(W_{1}\otimes W_{2})^{*}$ defined by
\begin{equation}\label{lambda-I-w-3-'}
\lambda_{I, w_{3}'}(w_{1}\otimes w_{2})
=\langle w_{3}', I(w_{1}\otimes w_{2})\rangle
\end{equation}
for $w_{1}\in W_{1}$ and $w_{2}\in W_{2}$. 
Let $W_{1}\hboxtr_{P(z)}W_{2}$ be the subspace of 
 $(W_{1}\otimes W_{2})^{*}$ spanned by
$\lambda_{I, w_{3}'}$ for all $P(z)$-products $(W_{3}, I)$
and $w_{3}'\in W_{3}'$. We define 
$(W_{1}\hboxtr_{P(z)}W_{2})_{[n]}^{[\alpha]}$ to be the subspace of $W_{1}\hboxtr_{P(z)}W_{2}$
spanned by $\lambda_{I, w_{3}'}$ for all $P(z)$-products $(W_{3}, I)$
and $w_{3}'\in w_{3}'\in (W_{3}')_{[n]}^{[\alpha]}\subset W_{3}'$ 
for $n\in \C$ and $\alpha\in P_{W_{3}'}^{g}$. Then 
$$W_{1}\hboxtr_{P(z)}W_{2}=\coprod_{n\in \C}\coprod_{\alpha+\Z\in \C/\Z}
(W_{1}\hboxtr_{P(z)}W_{2})_{[n]}^{[\alpha]}.$$

We define a vertex operator map 
$$Y_{W_{1}\hboxtr_{P(z)}W_{2}}:
V\otimes (W_{1}\hboxtr_{P(z)}W_{2})\to 
(W_{1}\hboxtr_{P(z)}W_{2})\{x\}[\log x]$$
by
\begin{equation}\label{Y-W-1-W_2}
Y_{W_{1}\hboxtr_{P(z)}W_{2}}(v, x)\lambda_{I, w_{3}'}
=\lambda_{I, Y_{W_{3}'}
(v, x)w_{3}'}
\end{equation}
for $v\in V$ and $\lambda_{I, w_{3}'}\in W_{1}\hboxtr_{P(z)}W_{2}$.
We recall the following result in  \cite{DH}:

\begin{prop}[Proposition 4.3 in \cite{DH}] 
The pair $(W_{1}\hboxtr_{P(z)}W_{2}, 
Y^{g_{1}g_{2}}_{W_{1}\hboxtr_{P(z)}W_{2}})$ is a generalized 
$(g_{1}g_{2})^{-1}$-twisted $V$-module. 
\end{prop}

Note that even though we have not 
shown that $W_{1}\hboxtr_{P(z)}W_{2}$ is grading-restricted yet,  its contragredient 
$(W_{1}\hboxtr_{P(z)}W_{2})'$ is still well defined. 
Let $W_{1}\boxtimes_{P(z)}W_{2}=(W_{1}\hboxtr_{P(z)}W_{2})'$.
The results and proofs in \cite{DH} in fact shows that $\boxtimes_{P(z)}$ 
gives a functor from 
$\mathcal{C}_{n}^{G}\times \mathcal{C}_{n}^{G}$ to the category of generalized 
$g$-twisted $V$-module for $g\in G$. We want to show in particular that the image of this functor 
is  in $\mathcal{C}_{n}^{G}$. 

We have the candidate  $W_{1}\boxtimes_{P(z)}W_{2}$ for our $P(z)$-tensor product 
twisted module.
To obtain  a $P(z)$-tensor product of $W_{1}$ and $W_{2}$, we also need 
a $P(z)$-intertwining map $\boxtimes_{P(z)}$. 
We recall the construction of $\boxtimes_{P(z)}$ in \cite{DH}. 

The construction of $\boxtimes_{P(z)}$ in \cite{DH} is given after Assumption 4.4 in \cite{DH}.
But this construction 
works without Assumption 4.4. In fact,  Assumption 4.4 in \cite{DH} contains
three assumptions. In our case, the first and second assumptions in Assumption 4.4 in \cite{DH}
do not hold. But we will still be able to prove below 
that $W_{1}\boxtimes_{P(z)}W_{2}$ is a $P(z)$-tensor product of $W_{1}$ and $W_{2}$ in
the category $\mathcal{C}_{n}^{G}$. 

We need Proposition 4.5 in \cite{DH}. Let $W$ be a generalized 
$(g_{1}g_{2})^{-1}$-twisted $V$-module
and $f: W\to W_{1}\hboxtr_{P(z)}W_{2}$ a $V$-module map. Note that in \cite{DH},
$W$ is required to be in the category $\mathcal{C}$ considered there. 
But what we know now is that  $W_{1}\boxtimes_{P(z)}W_{2}$ is 
 a generalized $(g_{1}g_{2})^{-1}$-twisted $V$-module. 
In this section, we will use Proposition 4.5 in \cite{DH} to prove that 
 $W_{1}\boxtimes_{P(z)}W_{2}$ is $C_{n}$-cofinite grading-restricted generalized 
$(g_{1}g_{2})^{-1}$-twisted $V$-module. So we cannot assume that $W$ is an object of 
$\mathcal{C}_{n}^{G}$. 
But the construction and Proposition 4.5 in \cite{DH} works even when we do not know whether 
$W$ is an object of 
$\mathcal{C}_{n}^{G}$. 

Since the double contragredient of $W$ might not be equivalent to 
$W$, we cannot view elements of $W''$ as elements of $W$. But we 
know that an element of $W'$ is given by its pairing with all elements of $W$.
We write the pairing between $W$ and $W'$ by $\langle w, w'\rangle$ 
instead of $\langle w', w\rangle$
for $w\in W$ and $w'\in W'$. 
We define a linear map
\begin{align*}
f': W_{1}\otimes W_{2}&\to \overline{W'}=\prod_{n\in \C}W_{[n]}^{*}\nn
w_{1}\otimes w_{2}&\mapsto \Y_{f}(w_{1}, z)w_{2}
\end{align*} 
by
\begin{equation}\label{int-op-z}
\langle w, f'(w_{1}\otimes w_{2})\rangle
=(f(w))(w_{1}\otimes w_{2})
\end{equation}
for $w_{1}\in W_{1}$, $w_{2}\in W_{2}$ and 
$w\in W$. We then define another  linear map 
\begin{align*}
\Y_{f}: W_{1}\otimes W_{2}&\to W'\{x\}[\log x]\nn
w_{1}\otimes w_{2}&\mapsto f'(w_{1}\otimes w_{2})
\end{align*} 
by
\begin{equation}\label{int-op-x}
\Y_{f}(w_{1}, x)w_{2}=x^{L_{W'}(0)}
e^{-(\log z) L_{W_{3}}(0)}
\Y_{f}(x^{-L_{W_{1}}(0)}e^{(\log z) L_{W_{1}}(0)}w_{1}, z)
x^{-L_{W_{2}}(0)}e^{(\log z) L_{W_{2}}(0)}w_{2}
\end{equation}
for $w_{1}\in W_{1}$ and $w_{2}\in W_{2}$. 
Note that by our notation, $f'=\Y_{f}(\cdot z)\cdot$.

\begin{prop}[Proposition 4.5 in \cite{DH}]\label{mod-map-int-op}
The linear map $f'=\Y_{f}(\cdot z)\cdot$ given by \eqref{int-op-z}
and 
$\Y_{f}: W_{1}\otimes W_{2}\to W'\{x\}[\log x]$
given by \eqref{int-op-z} and \eqref{int-op-x}
are a twisted $P(z)$-intertwining map and 
 a twisted intertwining operator, respectively, 
of type $\binom{W'}{W_{1}W_{2}}$. In particular, 
in the case that $W=W_{1}\hboxtr_{P(z)}W_{2}$ and 
$f=1_{W_{1}\hboxtr_{P(z)}W_{2}}: W\to W_{1}\hboxtr_{P(z)}W_{2}$
is the identity map, we obtain a  twisted $P(z)$-intertwining map
$1_{W_{1}\hboxtr_{P(z)}W_{2}}'$ and a twisted 
intertwining operator $\Y_{1_{W_{1}\hboxtr_{P(z)}W_{2}}}$
of type $\binom{W_{1}\boxtimes_{P(z)}W_{2}}{W_{1}W_{2}}$.
\end{prop}

We denote the $P(z)$-intertwining map of type $\binom{W_{1}\boxtimes_{P(z)}W_{2}}
{W_{1}W_{2}}$ in Proposition \ref{mod-map-int-op} by $\boxtimes_{P(z)}$. Let 
$$w_{1}\boxtimes_{P(z)}w_{2n}=\boxtimes_{P(z)}(w_{1}\otimes w_{2})
=\Y(w_{1}, z)w_{2}\in \overline{W_{1}\boxtimes_{P(z)}W_{2}}$$ 
for $w_{1}\in W_{1}$ and 
$w_{2}\in W_{2}$. The element $w_{1}\boxtimes_{P(z)}w_{2}$ is the 
tensor product 
of the elements $w_{1}$ and $w_{2}$. By \eqref{int-op-z}, we have
\begin{equation}\label{int-op-x-6}
\lambda(w_{1}\otimes w_{2})
=\langle \lambda, w_{1}\boxtimes_{P(z)}w_{2}\rangle
\end{equation}
for $\lambda\in W_{1}\hboxtr_{P(z)}W_{2}$, $w_{1}\in W_{1}$ and 
$w_{2}\in W_{2}$.

We also need a result on lower-bounded
generalized   surjective  product of two $C_{n}$-cofinite grading-restricted generalized
 $V$-modules.  

\begin{prop}\label{gradings-submod}
Let $g_{1}$ and $g_{2}$ be commuting automorphisms of $V$ and
 $W_{1}$ and $W_{2}$ $C_{n}$-cofinite grading-restricted generalized $g_{1}$- and $g_{2}$-twisted 
 $V$-modules, respectively.  Then 
the lower-bounded generalized surjective products of $W_{1}$ and $W_{2}$ 
are uniformly lower bounded, that is, there exists $N_{0}\in \Z$ depending only on $W_{1}$ and $W_{2}$
 such that for any lowr-bounded
generalized surjective product $W_{3}$ of $W_{1}$ and $W_{2}$,  
$(W_{3})_{[m]}=0$ when $\Re(m)<N_{0}$. 
\end{prop}
\begin{proof}
We consider the set of $\dim (W_{3}/C_{1}(W_{3}))$ for all lower-bounded
generalized surjective product $W_{3}$ of $W_{1}$ and $W_{2}$.
By Theorem \ref{n-W-1-W-2},
this set of nonnegative integers 
is bounded from above by $\dim (W_{1}/C_{1}(W_{1}))\dim  (W_{2}/C_{1}(W_{2}))$. 
Hence there must be a maximum of this set. 
Let $W_{3}^{\mathrm{max}}$ be a lower-bounded
generalized surjective product  of $W_{1}$ and $W_{2}$
such that $\dim W_{3}/C_{1}(W_{3})$
is the maximum of the set. Since $W_{3}^{\mathrm{max}}$ is 
lower-bounded, there exists $N_{0}\in \Z$ such that 
$(W_{3}^{\mathrm{max}})_{[n]}=0$ when $\Re(n)<N_{0}$. 

Given any lower-bounded
generalized surjective product $W_{3}$ of $W_{1}$ and $W_{2}$, we have a 
surjective intertwining operator $\Y$ of type $\binom{W_{3}}{W_{1}W_{2}}$.
For any fixed $z\in \C^{\times}$, we have a $P(z)$-intertwining map 
$I_{\Y}=\Y(\cdot, z)\cdot$ of type $\binom{W_{3}}{W_{1}W_{2}}$.
The $P(z)$-intertwining map $I_{\Y}$ gives 
a $V$-module map $I_{\Y}'$ from $W_{3}'$ to 
the generalized $(g_{1}g_{2})^{-1}$-twisted $V$-module $W_{1}\hboxtr_{P(z)}W_{2}$
defined by 
$$(I_{\Y}'(w_{3}'))(w_{1}\otimes w_{2})=\langle w_{3}', I_{\Y}(w_{1}\otimes w_{2})\rangle
=\langle w_{3}', \Y(w_{1}, z)w_{2}\rangle$$
for $w_{1}\in W_{1}$, $w_{2}\in W_{2}$, and $w_{3}'\in W_{3}'$.
Since the intertwining operator $\Y$ is surjective, this $V$-module 
map is injective.
The image of $W_{3}'$ under this $V$-module map 
is a lower-bounded generalized $V$-submdoule of 
$W_{1}\hboxtr_{P(z)}W_{2}$.  The same is also true for 
$W_{3}^{\mathrm{max}}$.
Let $W$ be the sum of the image of $W_{3}'$ under
the $V$-module map from $W_{3}'$ to  $W_{1}\hboxtr_{P(z)}W_{2}$ 
and the image of $(W_{3}^{\mathrm{max}})'$ under 
the $V$-module map from $(W_{3}^{\mathrm{max}})'$ to  
$W_{1}\hboxtr_{P(z)}W_{2}$. Then $W$ is 
also a lower-bounded generalized $V$-submodule
of $W_{1}\hboxtr_{P(z)}W_{2}$. Let 
$J: W\to W_{1}\hboxtr_{P(z)}W_{2}$ be the inclusion map. 
Then by Proposition \ref{mod-map-int-op}, we have a twisted intertwining operator 
$\Y_{J}$
of type $\binom{W'}{W_{1}W_{2}}$. Since $J$ is injective,  $\Y_{J'}$
is  surjective. 
Since $W$ is lower bounded, $W'$ is also 
lower bounded. By Theorem \ref{n-W-1-W-2},
$W'$ is a $C_{n}$-cofinite grading-restricted 
generalized $g_{1}g_{2}$-twisted $V$-module and
$\dim W'/C_{n}(W')\le 
\dim W_{3}^{\mathrm{max}}/C_{1}(W_{3}^{\mathrm{max}})$. 

We now prove 
$(W')_{[m]}=0$ when $\Re(m)<N_{0}$. 
By definition, we have an injective $V$-module map 
from $(W_{3}^{\mathrm{max}})'$ to $W$. 
Its adjoint is a surjective $V$-module map 
$f: W'\to W_{3}^{\mathrm{max}}$.
Since $f$ is a $V$-module map, 
$f(C_{n}(W'))\subset C_{n}(W_{3}^{\mathrm{max}})$. 
Then the map $f$ induces a surjective linear map 
$\bar{f}: W'/C_{n}(W')
\to W_{3}^{\mathrm{max}}/C_{n}(W_{3}^{\mathrm{max}})$. In particular, we have 
$\dim (W'/C_{n}(W'))\ge 
\dim (W_{3}^{\mathrm{max}}/C_{n}(W_{3}^{\mathrm{max}}))$. But we already 
have $\dim (W'/C_{n}(W'))\le
\dim (W_{3}^{\mathrm{max}}/C_{n}(W_{3}^{\mathrm{max}}))$. So we obtain 
$$\dim (W'/C_{n}(W'))=
\dim (W_{3}^{\mathrm{max}}/C_{n}(W_{3}^{\mathrm{max}})).$$ 
Thus $\bar{f}$ is injective and 
$\ker \bar{f}=0$. Since $\bar{f}$ is induced from $f$, we 
obtain $\ker f\subset C_{n}(W')$. 

If $(W')_{[m]}\ne 0$ for some $m\in \C$ satisfying $\Re(m)<N_{0}$, then we can find $m_{0}\in \C$
such that $\Re(m_{0})<N_{0}$, $(W')_{m_{0}]}\ne 0$, and $(W')_{[m]]}=0$ for $\Re(m)< m_{0}-1$ 
since $W'$ is lower-bounded. 
Let $w'\in W_{[m_{0}]}^{*}$. 
Then $f(w')\in (W_{3}^{\mathrm{max}})_{[m_{0}]}$. Since $\Re(m_{0})<N_{0}$, we have $f(w')=0$
from the definition of $N_{0}$. Then $w'\in \ker f\subset C_{n}(W')$. Then $w'$ 
is a linear combination of elements of the form 
$v_{\alpha-n}\tilde{w}'$ for homogeneous $v\in V^{[\alpha]}$ and homogeneous $\tilde{w}'\in W'$.
But $\wt v_{\alpha-n}\tilde{w}'=\wt v-\alpha+n-1+\wt \tilde{w}'$. Since $w'$ is a linear combination
of such elements, we must have $m_{0}=\wt v_{\alpha-n}\tilde{w}'=\wt v-\alpha+n-1+\wt \tilde{w}'$. 
Then we have 
$\wt \tilde{w}'=m_{0}-\wt v+\alpha-n+1<m_{0}-1$. Since $(W')_{[m]]}=0$ for $\Re(m)< m_{0}-1$, 
we obtain $\tilde{w}'=0$ and therefore $w'=0$. Thus we obtain $(W')_{m_{0}]}=0$. Contradiction. 
So we have proved that  $(W')_{[m]}=0$ when $\Re(m)<N_{0}$. But 
$(W')_{[m]}=0$ when $\Re(m)<N_{0}$ implies $W_{[m]}=0$ when $\Re(m)<N_{0}$.

Since the image $I_{\Y}'(W_{3}')$ of $W_{3}'$ under $I_{\Y}'$ in $W_{1}\hboxtr_{P(z)}W_{2}$ is in $W$, 
we also have $(I_{\Y}'(W_{3}'))_{[m]}=0$ when $\Re(m)<N_{0}$. Since $\Y$ is surjective, $I_{\Y}'$
is injective. So we also have $(W_{3}')_{[m]}=0$ when $\Re(m)<N_{0}$. Thus we obtain 
$(W_{3})_{[m]}=0$ when $\Re(m)<N_{0}$.
\end{proof}

\begin{thm}\label{hboxtr-closed}
Let $g_{1}, g_{2}\in G$, 
$W_{1}$, $W_{2}$, $W_{3}$ grading-restricted generalized $g_{1}$-, $g_{2}$-,
$g_{1}g_{2}$-twisted $V$-modules, respectively, in the category $\mathcal{C}_{n}^{G}$
and  $z\in \C^{\times}$. Then $W_{1}\hboxtr_{P(z)}W_{2}$ is a 
grading-restricted generalized $(g_{1}g_{2})^{-1}$-twisted $V$-module and 
$W_{1}\boxtimes_{P(z)}W_{2}$ is a $C_{n}$-cofinite grading-restricted generalized 
$g$-twisted $V$-module.  Moreover, $(W_{1}\boxtimes_{P(z)}W_{2}, \boxtimes_{P(z)})$ is a 
$P(z)$-tensor product of $W_{1}$ and $W_{2}$ in the category $\mathcal{C}_{n}^{G}$. 
\end{thm}
\begin{proof}
For any $P(z)$-product $(W_{3}, I)$, we have a module map $I': W_{3}'\to 
W_{1}\hboxtr_{P(z)}W_{2}$ given by $(I'(w_{3}'))(w_{1}\otimes w_{2})=
\langle w_{3}', I(w_{1}\otimes w_{2})\rangle$ for $w_{3}'\in W_{3}'$, $w_{1}\in W_{1}$, and
$w_{2}\in W_{2}$. Consider the image $I'(W_{3})\in W_{1}\hboxtr_{P(z)}W_{2}$ of $W_{3}'$ under $I'$. 
By Proposition \ref{mod-map-int-op}, the inclusion map $J: I'(W_{3})\to W_{1}\hboxtr_{P(z)}W_{2}$ gives a 
$P(z)$-intertwining operator $\Y_{J}$
of type $\binom{I'(W_{3})'}{W_{1} W_{2}}$.   Since $J$ is injective,  $\Y_{J}$ is surjective
By Proposition \ref{gradings-submod}, $(I'(W_{3})')_{[m]}=0$ when $\Re(m)<N_{0}$,
where $N_{0}\in \Z$ depends only on $W_{1}$ and $W_{2}$ and its existence 
 is given by Proposition \ref{gradings-submod}.  Hence $I'(W_{3})$ has the same property, that is, 
$I'(W_{3})_{[m]}=0$ when $\Re(m)<N_{0}$. 

By definition, $W_{1}\hboxtr_{P(z)}W_{2}$ is in fact the sum of $I'(W_{3})$ for all $P(z)$-products
of the form $(W_{3}, I)$. From the property $I'(W_{3})_{[m]}=0$ when $\Re(m)<N_{0}$,
the sum of all  such $I'(W_{3})$ satisfies the same property, that is, 
$(W_{1}\hboxtr_{P(z)}W_{2})_{[m]}=0$ when $\Re(m)<N_{0}$. Thus $W_{1}\hboxtr_{P(z)}W_{2}$ 
is in fact lower-bounded. So $W_{1}\boxtimes_{P(z)}W_{2}$  is also lower-bounded. 

The same argument shows that the  twisted intertwining operator $\Y_{1_{W_{1}\hboxtr_{P(z)}W_{2}}}$ 
given in Proposition \ref{mod-map-int-op}
is surjective. By Theorem \ref{n-W-1-W-2}, we see that $W_{1}\boxtimes_{P(z)}W_{2}$ 
is $C_{n}$-cofinite. Then by Proposition \ref{C-2-gr-wk}, $W_{1}\boxtimes_{P(z)}W_{2}$  
is quasi-finite-dimensional, and, in particular, is grading-restricted. As the contragredient of 
$W_{1}\boxtimes_{P(z)}W_{2}$, $W_{1}\hboxtr_{P(z)}W_{2}$ is also grading-restricted. 

Finally, we prove that $W_{1}\boxtimes_{P(z)}W_{2}$ together with the twisted $P(z)$-intertwining map 
$\boxtimes_{P(z)}$ is a $P(z)$-tensor product of $W_{1}$ and $W_{2}$. Given a $P(z)$-product 
$(W_{3}, I)$, we have a module map $I': W_{3}'\to W_{1}\hboxtr_{P(z)}W_{2}$ defined above. 
Then the adjoint $\eta: W_{1}\boxtimes_{P(z)}W_{2}\to W_{3}$ of $I'$ is also a module map. 
Let $\bar{\eta}: \overline{W_{1}\boxtimes_{P(z)}W_{2}}\to \overline{W}_{3}$ be the natural extension of $\eta$ 
Then by definitions, we have 
\begin{align}\label{hboxtr-closed-2}
\langle w_{3}',  \bar{\eta}(w_{1}\boxtimes_{P(z)} w_{2})\rangle =
\langle I'(w_{3}'), w_{1}\boxtimes_{P(z)} w_{2})\rangle 
= (I'(w_{3}'))(w_{1}\otimes w_{2})=\langle w_{3}', I(w_{1}\otimes w_{2})\rangle
\end{align}
for $w_{1}\in W_{1}$, $w_{2}\in W_{2}$, and $w_{3}'\in W_{3}'$.
So we obtain $\bar{\eta}(w_{1}\boxtimes_{P(z)} w_{2})=I(w_{1}\otimes w_{2})$. 
Since $\bar{\eta}(w_{1}\boxtimes_{P(z)} w_{2})$ can be rewritten as 
$(\bar{\eta}\circ \boxtimes_{P(z)})(w_{1}\otimes w_{2})$, we obtain 
$\bar{\eta}\circ \boxtimes_{P(z)}=I$. The uniqueness of $\eta$ follows from the definition of $\eta$. 
We have proved the universal property for $(W_{1}\boxtimes_{P(z)}W_{2}, \boxtimes_{P(z)})$.
So $(W_{1}\boxtimes_{P(z)}W_{2}, \boxtimes_{P(z)})$ is a $P(z)$-tensor product of $W_{1}$ and $W_{2}$. 
\end{proof}

Theorem \ref{hboxtr-closed} in fact assigns each object $(W_{1}, W_{2})$ in the category 
$\mathcal{C}_{n}^{G}\times \mathcal{C}_{n}^{G}$ an object $W_{1}\boxtimes_{P(z)}W_{2}$
in $\mathcal{C}_{n}^{G}$. As in \cite{DH}, we can also assign a morphism $(f_{1}, f_{2})$ in 
$\mathcal{C}_{n}^{G}\times \mathcal{C}_{n}^{G}$ a morphism $f_{1}\boxtimes_{P(z)}f_{2}$
in $\mathcal{C}_{n}^{G}$ in the same way by using the universal property of 
a $P(z)$-tensor product. We then also obtain  the following result whose proof is the same as the proof 
of Theorem 4.7 in \cite{DH}:

\begin{thm}\label{P-z-bifunctor}
The assignment, denoted by $\boxtimes_{P(z)}$, given by $(W_{1}, W_{2})\mapsto W_{1}\boxtimes_{P(z)}W_{2}$
and $(f_{1}, f_{2})\mapsto f_{1}\boxtimes_{P(z)}f_{2}$ is a functor from 
$\mathcal{C}_{n}^{G}\times \mathcal{C}_{n}^{G}$ to $\mathcal{C}_{n}^{G}$.
\end{thm}

\subsection{$P(z)$-tensor product bifunctor for $\widetilde{\mathcal{C}}_{n}^{G}$}

In this subsection, for $z\in \C^{\times}$, 
we give a $P(z)$-tensor product bifunctor for the category $\widetilde{\mathcal{C}}_{n}^{G}$
 of $C_{n}$-cofinite grading-restricted generalized $g$-twisted $V$-modules
with $G$-actions for $g\in G$. 

Let $g_{1}, g_{2}\in G$. 
Given $C_{n}$-cofinite grading-restricted generalized $g_{1}$- and $g_{2}$-twisted $V$-modules
$W_{1}$ and $W_{2}$, respectively, for $g_{1}, g_{2}\in G$. 
Since $W_{1}$ and $W_{2}$ are $G$-modules, $(W_{1}\otimes W_{2})^{*}$ is also a $G$-module. 
In particular, $(W_{1}\otimes W_{2})^{*}$ has an action of $(g_{1}g_{2})^{-1}$. 

Given a $P(z)$-product $(W_{3}, I)$ of $W_{1}$ and $W_{2}$  in $\widetilde{\mathcal{C}}_{n}^{G}$,
for $w_{3}'\in W_{3}'$, we have $\lambda_{I, w_{3}'}\in 
(W_{1}\otimes W_{2})^{*}$ defined by \eqref{lambda-I-w-3-'}.
Let $W_{1}\hboxtr_{P(z)}W_{2}$ be the subspace of 
 $(W_{1}\otimes W_{2})^{*}$ spanned by
$\lambda_{I, w_{3}'}$ for all $P(z)$-products $(W_{3}, I)$
and $w_{3}'\in W_{3}'$. Note that in general, $W_{1}\hboxtr_{P(z)}W_{2}$ here
might be different from $W_{1}\hboxtr_{P(z)}W_{2}$ in the case that 
$W_{1}$ and $W_{2}$ are viewed as grading-restricted generalized $g$-twisted $V$-modules
without $g$ and $G$-actions. For simplicity, we use the same notation $\hboxtr_{P(z)}$, but we will call
them $W_{1}\hboxtr_{P(z)}W_{2}$ in $\widetilde{\mathcal{C}}_{n}^{G}$ and
$W_{1}\hboxtr_{P(z)}W_{2}$ in $\mathcal{C}_{n}^{G}$ to distinguish them when it is necessary. 

As discussed above, $G$ acts on $(W_{1}\otimes W_{2})^{*}$.
By the definition of twisted $P(z)$-intertwining map
in the category $\widetilde{\mathcal{C}}_{n}^{G}$, 
for a $P(z)$-product $(W_{3}, I)$ of $W_{1}$ and $W_{2}$  in $\widetilde{\mathcal{C}}_{n}^{G}$,
and  $w_{3}'\in (W_{3}')^{[\alpha]}$, we have 
\begin{align*}
(h\lambda_{I, w_{3}'})(w_{1}\otimes w_{2})
&=\lambda_{I, w_{3}'}(h^{-1}(w_{1}\otimes w_{2}))\nn
&=\lambda_{I, w_{3}'}(h^{-1}w_{1}\otimes h^{-1}w_{2})\nn
&=\langle w_{3}', I(h^{-1}w_{1}\otimes h^{-1}w_{2})\rangle\nn
&=\langle w_{3}', h^{-1}I(w_{1}\otimes w_{2})\rangle\nn
&=\langle hw_{3}', I(w_{1}\otimes w_{2})\rangle\nn
&=\lambda_{I, hw_{3}'}(w_{1}\otimes w_{2})
\end{align*}
for $w_{1}\in W_{1}$ and $w_{2}\in W_{2}$.  So we obtain 
$h\lambda_{I, w_{3}'}=\lambda_{I, hw_{3}'}$, which means that 
$h\lambda_{I, w_{3}'}\in W_{1}\hboxtr_{P(z)}W_{2}$. Thus  the  $G$-action on $(W_{1}\otimes W_{2})^{*}$
induces a $G$-action on $W_{1}\hboxtr_{P(z)}W_{2}$.

Taking $h=(g_{1}g_{2})^{-1}$ in $h\lambda_{I, w_{3}'}=\lambda_{I, hw_{3}'}$, we obtain
$((g_{1}g_{2})^{-1}\lambda_{I, w_{3}'})
=\lambda_{I, (g_{1}g_{2})^{-1}w_{3}'}$.
Then we have
$$((g_{1}g_{2})^{-1}-e^{2\pi \i\alpha})^{k}\lambda_{I, w_{3}'}
=\lambda_{I, (g_{1}g_{2})^{-1}-e^{2\pi \i\alpha})^{k}w_{3}'}$$
for $k\in \Z_{+}$. Since $w_{3}'\in (W_{3}')^{[\alpha]}$, there exists $k\in \Z_{+}$ such that
$((g_{1}g_{2})^{-1}-e^{2\pi \i\alpha})^{k}w_{3}'=0$. So  we obtain
$((g_{1}g_{2})^{-1}-e^{2\pi \i\alpha})^{k}\lambda_{I, w_{3}'}=0$. 
Let $(W_{1}\hboxtr_{P(z)}W_{2})^{[\alpha]}$ be the generalized eigenspace of 
the action of $(g_{1}g_{2})^{-1}$ on $W_{1}\hboxtr_{P(z)}W_{2}$ with eigenvalue 
$e^{2\pi \i\alpha}$. We have proved that $\lambda_{I, w_{3}'}\in (W_{1}\hboxtr_{P(z)}W_{2})^{[\alpha]}$. 
Since $W_{1}\hboxtr_{P(z)}W_{2}$ is spanned by elements of the form 
$\lambda_{I, w_{3}'}$, $W_{1}\hboxtr_{P(z)}W_{2}$ is a direct sum of 
generalized eigenspaces of the action of $(g_{1}g_{2})^{-1}$ on $W_{1}\hboxtr_{P(z)}W_{2}$,
that is, 
$$W_{1}\hboxtr_{P(z)}W_{2}=\coprod_{\alpha+\Z\in \C/\Z}
(W_{1}\hboxtr_{P(z)}W_{2})^{[\alpha]}.$$

We define a vertex operator map 
$$Y_{W_{1}\hboxtr_{P(z)}W_{2}}:
V\otimes (W_{1}\hboxtr_{P(z)}W_{2})\to 
(W_{1}\hboxtr_{P(z)}W_{2})\{x\}[\log x]$$
using the same formula \eqref{Y-W-1-W_2} as in the case 
in the category $\mathcal{C}_{n}^{G}$. 
We consider the identity operator $1_{W_{1}\hboxtr_{P(z)}W_{2}}: 
W_{1}\hboxtr_{P(z)}W_{2}\to W_{1}\hboxtr_{P(z)}W_{2}$ on 
$W_{1}\hboxtr_{P(z)}W_{2}$.
By taking $f=1_{W_{1}\hboxtr_{P(z)}W_{2}}$ in Proposition \ref{mod-map-int-op}, we obtain a 
twisted $P(z)$-intertwining map $1_{W_{1}\hboxtr_{P(z)}W_{2}}'$ in the category 
$\mathcal{C}_{n}^{G}$. But note that
$W_{1}\hboxtr_{P(z)}W_{2}$ in this subsection might be 
different from $W_{1}\hboxtr_{P(z)}W_{2}$ in the preceding subsection.
This twisted $P(z)$-intertwining map $\boxtimes_{P(z)}$ in the category $\mathcal{C}_{n}^{G}$
might also be different from the twisted $P(z)$-intertwining map $\boxtimes_{P(z)}$ obtained in 
the last part of Proposition \ref{mod-map-int-op}. For simplicity, we will use the same 
notation $\boxtimes_{P(z)}$ to denote $1_{W_{1}\hboxtr_{P(z)}W_{2}}'$. 
We also have $w_{1}\boxtimes_{P(z)} w_{2}$ of $w_{1}\in W_{1}$ and $w_{2}\in W_{2}$
defined to be $\boxtimes_{P(z)}(w_{1}\otimes w_{2})$. 
Proposition \ref{mod-map-int-op} does not tell us whether 
$\boxtimes_{P(z)}$ is a twisted $P(z)$-intertwining map  in the category $\widetilde{\mathcal{C}}_{n}^{G}$.
But we have 
\begin{align*}
\langle \lambda, h(w_{1}\boxtimes_{P(z)} w_{2})\rangle
&=\langle h^{-1}\lambda, w_{1}\boxtimes_{P(z)} w_{2}\rangle\nn
&=(h^{-1}\lambda)(w_{1}\otimes w_{2})\nn
&=\lambda(h(w_{1}\otimes w_{2}))\nn
&=\lambda(hw_{1}\otimes hw_{2})\nn
&=\langle \lambda, hw_{1}\boxtimes_{P(z)} hw_{2}\rangle
\end{align*}
for $\lambda \in W_{1}\hboxtr_{P(z)}W_{2}$ in $\widetilde{\mathcal{C}}_{n}^{G}$, 
$w_{1}\in W_{1}$, and $W_{2}\in W_{2}$. 
Thus we obtain 
$$h(w_{1}\boxtimes_{P(z)} w_{2})=hw_{1}\boxtimes_{P(z)} hw_{2}$$
for $h\in G$, $w_{1}\in W_{1}$ and $w_{2}\in W_{2}$. 
So  $\boxtimes_{P(z)}$  is in fact a twisted $P(z)$-intertwining map 
in the category $\widetilde{\mathcal{C}}_{n}^{G}$. In particular, the twisted 
intertwining operator $\Y_{1_{W_{1}\hboxtr_{P(z)}W_{2}}}$ is a twisted intertwining operator 
in the category $\widetilde{\mathcal{C}}_{n}^{G}$. 

\begin{thm}\label{hboxtr-closed-tilde-C}
Let $g_{1}, g_{2}$ be commuting automorphisms of $V$, 
$W_{1}$, $W_{2}$, $W_{3}$ grading-restricted generalized $g_{1}$-, $g_{2}$-,
$g_{1}g_{2}$-twisted $V$-modules, respectively, in the category $\widetilde{\mathcal{C}}_{n}^{G}$
and  $z\in \C^{\times}$. Then $W_{1}\hboxtr_{P(z)}W_{2}$ is a 
grading-restricted generalized $(g_{1}g_{2})^{-1}$-twisted $V$-module with 
a $G$-action and its contragredient 
$W_{1}\boxtimes_{P(z)}W_{2}=(W_{1}\hboxtr_{P(z)}W_{2})'$ is a $C_{n}$-cofinite grading-restricted generalized 
$g$-twisted $V$-module with a $G$-action.  Moreover, $(W_{1}\boxtimes_{P(z)}W_{2}, \boxtimes_{P(z)})$ is a 
$P(z)$-tensor product of $W_{1}$ and $W_{2}$ in the category $\widetilde{\mathcal{C}}_{n}^{G}$. 
\end{thm}
\begin{proof}
Since twisted $P(z)$-intertwining maps 
in the category $\widetilde{\mathcal{C}}_{n}^{G}$ are also twisted $P(z)$-intertwining maps 
in the category $\mathcal{C}_{n}^{G}$, $W_{1}\hboxtr_{P(z)}W_{2}$ in 
this subsection is a generalized $(g_{1}g_{1})^{-1}$-twisted 
$V$-submodule of what is denoted using the same notation 
$W_{1}\hboxtr_{P(z)}W_{2}$ in the preceding subsection. By Theorem \ref{hboxtr-closed},
$W_{1}\hboxtr_{P(z)}W_{2}$ in  the preceding subsection is grading-restricted, its 
generalized $(g_{1}g_{1})^{-1}$-twisted 
$V$-submodule $W_{1}\hboxtr_{P(z)}W_{2}$ in this subsection is also grading-restricted. 
Thus its contragredient $W_{1}\boxtimes_{P(z)}W_{2}$ is also grading-restricted. 

Since $1_{W_{1}\hboxtr_{P(z)}W_{2}}$ is injective, the twisted intertwining operator 
$\Y_{1_{W_{1}\hboxtr_{P(z)}W_{2}}}$ is surjective. Thus by Theorem \ref{n-W-1-W-2},
$W_{1}\boxtimes_{P(z)}W_{2}$ is $C_{n}$-cofinite and is thus in $\widetilde{\mathcal{C}}_{n}^{G}$.
Now the same argument as in the last part of the proof of 
Theorem \ref{hboxtr-closed} shows that 
$(W_{1}\boxtimes_{P(z)}W_{2}, \boxtimes_{P(z)})$ is a 
$P(z)$-tensor product of $W_{1}$ and $W_{2}$ in the category $\widetilde{\mathcal{C}}_{n}^{G}$. 
\end{proof}

We can also define $f_{1}\boxtimes_{P(x)}f_{2}$ for two morphisms $f_{1}$ and $f_{2}$ in 
$\widetilde{\mathcal{C}}_{n}^{G}$. 
Then the following result is obtained in the same way as Theorem \ref{P-z-bifunctor}:

\begin{thm}
The assignment, denoted by $\boxtimes_{P(z)}$, given by $(W_{1}, W_{2})\mapsto W_{1}\boxtimes_{P(z)}W_{2}$
and $(f_{1}, f_{2})\mapsto f_{1}\boxtimes_{P(z)}f_{2}$ is a functor from 
$\widetilde{\mathcal{C}}_{n}^{G}\times \widetilde{\mathcal{C}}_{n}^{G}$ to $\widetilde{\mathcal{C}}_{n}^{G}$.
\end{thm}

\noindent {\small \sc Department of Mathematics, Rutgers University,
110 Frelinghuysen Rd., Piscataway, NJ 08854-8019}

\noindent {\em E-mail address}: yzhuang@math.rutgers.edu

\end{document}